\documentclass[12pt,a4paper]{amsart}
\usepackage{amsmath,amsthm,amsfonts,amssymb,mathrsfs}
\usepackage[hypertexnames=false, colorlinks=true, linkcolor = teal, citecolor = blue]{hyperref}

 \usepackage[textwidth=1.8cm]{todonotes}
 \usepackage{marginnote}

\newcommand\R{\mathbb{R}}

\renewcommand{\d}{{{\rm d}}}

\newtheorem{theorem}{Theorem}
\newtheorem{cor}[theorem]{Corollary}

\newtheorem{lemma}[theorem]{Lemma}
\newtheorem{prop}[theorem]{Proposition}
\newtheorem{definition}[theorem]{Definition}

\theoremstyle{remark}
\newtheorem{rem}[theorem]{Remark}

\numberwithin{equation}{section}
\numberwithin{theorem}{section}
\usepackage{mathtools}
\mathtoolsset{showonlyrefs=true}

\title[Aggregation-diffusion equation]{Aggregation--Diffusion Equations with Regular \\
or Repulsive Interactions:
\\ Gaussian and Self-Similar Asymptotics}

\author[G. Karch]{Grzegorz Karch} 
\address[G. Karch]{	
Instytut Matematyczny, Uniwersytet Wroc\l{}awski, pl. Grunwaldzki 2, \hbox{50-384} Wroc\l{}aw, Poland \\ 
\href{https://orcid.org/0000-0001-9390-5578}{orcid.org/0000-0001-9390-5578}}
\email{grzegorz.karch@math.uni.wroc.pl}
\urladdr {http://karch.math.uni.wroc.pl}

\author[Y.~Soga]{Yuri Soga} 
\address[Y. Soga]{	
Mathematical Institute,
Tohoku University,
Sendai 980-8578,
Japan} 
\email{soga.yuri.q6@dc.tohoku.ac.jp}

\subjclass[2020]{Primary 35K55; Secondary 35B40, 35C06}

\keywords{aggregation-diffusion equation, asymptotic behavior, self-similar solutions}
\date{\today}
\begin{document}

\begin{abstract}
We study the large-time behavior of solutions to the aggregation--diffusion equation
$$
u_t=\Delta u+\nabla\cdot\bigl(u\,\nabla K*u\bigr)
\qquad \text{in } \mathbb{R}^d.
$$
Our main results identify a transition, governed by the regularity and singularity of the interaction kernel, between Gaussian and nonlinear self-similar asymptotics. We prove that for regular interaction kernels, as well as for sufficiently mild singular repulsive kernels, the nonlinear drift is asymptotically negligible and solutions exhibit the same Gaussian large-time behavior as solutions of the linear heat equation. In contrast, for the critical logarithmic kernel
$
K(x)=-\log |x|,
$
the interaction persists at the diffusive scale and leads to genuinely nonlinear asymptotics. More precisely, for every mass \(M>0\), we prove the existence and uniqueness of a self-similar solution of mass \(M\) and show that every solution with mass $M$ and with finite second moment converges in all $L^p$-norms to this self-similar solution as \(t\to\infty\).
\end{abstract}

\maketitle
\vspace{-0.7cm}
\tableofcontents

\section{Introduction}

\subsection{Aggregation-diffusion equations} 
We investigate the large-time behavior of solutions to the Cauchy problem
\begin{equation}\label{eq:main}
	\begin{cases}
		u_t =  \Delta u + \nabla \cdot (u \nabla K \ast u),
		& \quad x \in \R^d, \ t > 0,\\
		u(x,0) = u_0(x), & \quad x \in \R^d,
	\end{cases}
\end{equation}
in arbitrary spatial dimensions $d \ge 1$, subject to a non-negative initial condition
\begin{equation}\label{ini:L1Linf}
u_0 \in L^1(\R^d) \cap L^\infty(\R^d), \quad u_0 \ge 0.
\end{equation}
This model describes the evolution of a particle density $u = u(x, t)$ interacting through pairwise potentials, characterized by convolution with a kernel $K : \mathbb{R}^d \to \mathbb{R}$. Such models arise in several areas of science.
In astrophysics, mean-field models for gravitationally attracting particles are often rooted in the Chandrasekhar equation for stellar equilibrium \cite{Chandrasekhar1942,Chavanis1996}. Similarly, in mathematical biology, variations of the Keller--Segel system are utilized to describe population dynamics in phenomena such as chemotaxis, haptotaxis, and angiogenesis \cite{HillenPainter2009}.
For a comprehensive survey of recent results on equation \eqref{eq:main} (involving linear or nonlinear diffusion), we refer to \cite{Carrillo2019,Gomez-Castro24} and the references therein. In the specific case of linear diffusion, equation \eqref{eq:main} is often referred to as the McKean--Vlasov equation, which is associated with a stochastic differential equation via a mean-field limit procedure \cite{Jabin2017,Bresch2019}.

\subsection{Self-similar large-time behavior}

The main goal of this work is to identify a class of kernels $K = K(x)$ in model \eqref{eq:main} such that the corresponding solutions are global-in-time and exhibit self-similar large-time behavior. To set the context for our results on system \eqref{eq:main}--\eqref{ini:L1Linf} as $t \to \infty$, we first recall some classical PDE models where self-similar solutions play a fundamental role in governing long-time asymptotics.

\subsubsection{Heat equation}
The solution of the Cauchy problem for the heat equation is given by
the heat semigroup defined by
\begin{equation}\label{heat semigroup}
e^{t\Delta}u_0(x) = \int_{\mathbb{R}^d} G(x-y, t)u_0(y) \, \mathrm{d}y, 
\end{equation}
where 
\begin{equation}\label{heat kernel}
    G(x,t) = \frac{1}{(4\pi t)^{\frac{d}{2}}} \exp\left(-\frac{|x|^2}{4t}\right)
\end{equation}
is the standard heat kernel. If $u_0 \in L^1(\mathbb{R}^d)$, then the following asymptotic simplification holds true for every $p\in [1,\infty]$
\begin{equation}\label{heat:kernel:asymptotics}
t^{\frac{d}{2}\left(1-\frac{1}{p}\right)} \left\| e^{t\Delta}u_0(\cdot) - \int_{\mathbb{R}^d} u_0(x) \, \mathrm{d}x \, G(\cdot, t) \right\|_p \to 0 \quad\text{as}\quad t\to\infty.
\end{equation} 
Since $\|G(\cdot, t)\|_p=t^{-(d/2)(1-1/p)}\|G(\cdot,1)\|_p$ for all $t>0$,
this is indeed a first-order  asymptotic expansion of the heat semigroup \eqref{heat semigroup}.
The proof of this asymptotic result is based on a  Taylor expansion of the heat kernel $G(x,t)$ and analogous higher-order expansions can be found in the paper \cite{DZ92}.

\subsubsection{Convection-diffusion equation}
Analogous self-similar asymptotics are known for several nonlinear PDEs. We recall the closely related result of Escobedo and Zuazua
 \cite{EZ91}
for the convection-diffusion equation
\begin{equation}\label{eq:conv-diff}
u_t - \Delta u + \overline{a} \cdot \nabla (|u|^{q-1}u) = 0, \qquad x \in \mathbb{R}^d, \quad t > 0,
\end{equation}
where  $\overline{a} \in \mathbb{R}^d \setminus \{0\}$ is constant and 
$u(\cdot, 0)=u_0\in L^1(\R^d)$.
The result by Escobedo and Zuazua can be summarized as follows:
\begin{itemize}
    \item If $q > 1 + 1/d$, the solution to the Cauchy problem for equation \eqref{eq:conv-diff} asymptotically behaves like the heat kernel \eqref{heat kernel}:
\begin{equation}
    t^{\frac{d}{2}(1-\frac{1}{p})} \|u(\cdot, t) - M G(\cdot, t)\|_p \to 0 \quad \text{as } t \to \infty
\end{equation}
for every $p \in [1,\infty]$, where $M = \int_{\mathbb{R}^d} u_0(x) \, \mathrm{d}x$.

    \item If $q = 1 + 1/d$, there exists a self-similar solution $U_M(x,t)$ to equation \eqref{eq:conv-diff} such that
    $\int_{\R^d}U_M(x,t)\, \d x=M$ and 
\begin{equation}\label{conv-diff-self}
    t^{\frac{d}{2}(1-\frac{1}{p})} \|u(\cdot, t) - U_M(\cdot, t)\|_p \to 0 \quad \text{as } t \to \infty
\end{equation}
for every $p \in [1,\infty]$.
\end{itemize}

\subsubsection{Aggregation equation in one dimension}
Inspired by the asymptotic results for the con\-vec\-tion-diffusion equation \eqref{eq:conv-diff} mentioned above, we aim to obtain similar asymptotic expansions for the aggregation-diffusion model \eqref{eq:main}. The first step in this direction was taken by the authors of \cite{KS10-spikes}, who investigated the large-time behavior of solutions to the one-dimensional counterpart of problem \eqref{eq:main}:
\begin{equation} \label{eq:1.1}
\begin{cases}
    u_t = u_{xx} - (u(K' * u))_x &\quad \text{for } x \in \mathbb{R}, \; t > 0, \\
    u(x, 0) = u_0(x) &\quad \text{for } x \in \mathbb{R},
\end{cases}
\end{equation}
where $K' \in L^1(\mathbb{R})$ 
(notice that $K'$ in the paper \cite{KS10-spikes} corresponds to $-\partial_xK$ in this work)
and $u_0 \in L^1(\mathbb{R})$ is assumed to be nonnegative.
The primary result of \cite{KS10-spikes} establishes that solutions $u(x,t)$ to problem \eqref{eq:1.1} develop a self-similar asymptotic profile as $t \to \infty$. This profile depends entirely on the integral of the kernel $K'(x)$ and there are two distinct asymptotic regimes:
\begin{itemize}
    \item 
If $\int_{\mathbb{R}} K'(x) \, \mathrm{d}x = 0$, the solution to problem \eqref{eq:1.1} asymptotically behaves like the heat kernel \eqref{heat kernel}: for every $p \in [1, \infty]$, 
\begin{equation}
    t^{\frac{1}{2}(1-\frac{1}{p})} \|u(\cdot, t) - M {G}(\cdot, t)\|_p \to 0 \quad \text{as } t \to \infty, \label{eq:2.5}
\end{equation}
  where $M = \int_{\mathbb{R}} u_0(x) \, \mathrm{d}x$.

\item On the other hand, if $A \equiv \int_{\mathbb{R}} K'(x) \, \mathrm{d}x \neq 0$,  the solution behaves like a nonlinear diffusion wave. Namely, for every $p \in [1, \infty]$, 
\begin{equation}\label{selfsimliar-one-dimensional}
    t^{\frac{1}{2}(1-\frac{1}{p})} \|u(\cdot, t) - {U}_{M,A}(\cdot, t)\|_p \to 0 \quad \text{as } t \to \infty. 
\end{equation}
Here, ${U}_{M,A}(x,t) = t^{-1/2}{U}_{M,A}\left({x}/{{t^{1/2}}}, 1\right)$ is the unique self-similar solution to the initial value problem for the viscous Burgers equation
\begin{equation}
\begin{cases}
    U_t = U_{xx} - A(U^2)_x, &\quad \text{for } x \in \mathbb{R}, \; t > 0, \\
    U(x, 0) = M \delta_0, &
    \end{cases}
\end{equation}
where $\delta_0$ denotes the Dirac measure.
\end{itemize}

\subsubsection{Brief summary of results from this work}
Because the asymptotics of solutions to model \eqref{eq:main}--\eqref{ini:L1Linf} in one spatial dimension were described in \cite{KS10-spikes} (as noted above), in what follows, we systematically analyze this problem assuming the spatial dimension is $d \ge 2$.
We summarize the results here and state them precisely in the next section.
\begin{itemize}
    \item 
If the interaction kernel $K = K(x)$ in problem \eqref{eq:main}--\eqref{ini:L1Linf} is a regular function, meaning that both $K$ and $\nabla K$ are bounded and decay sufficiently fast as $|x| \to \infty$, then the self-similar large-time asymptotics of solutions to problem \eqref{eq:main}--\eqref{ini:L1Linf} are given by the heat kernel \eqref{heat kernel} (see Theorem \ref{thm:regular:asymptotics:intro}).

\item
The Gaussian asymptotics (as in formula \eqref{heat:kernel:asymptotics}) also appear in the case of solutions to problem \eqref{eq:main}--\eqref{ini:L1Linf} with singular, but not too singular, kernels. Here, the repulsive Riesz kernel 
$K(x) = A|x|^{-k}$ with $A>0$ and $k \in (0, d-2)$ serves as a canonical example, see Theorem \ref{thm:heat:kernel:intro}
and Remark \ref{examples}.

\item
Finally, we consider the repulsive kernel with the critical singularity $K(x) =- \log|x|$, for which the asymptotics are essentially different. In this case, equation \eqref{eq:main} has a class of self-similar solutions that describe the large-time behavior of other solutions to the Cauchy problem \eqref{eq:main}--\eqref{ini:L1Linf}, analogous to the critical cases for the convection-diffusion equation \eqref{conv-diff-self} and the one-dimensional aggregation-diffusion equation \eqref{selfsimliar-one-dimensional}. We refer to Theorem \ref{main thm:self-similar} and the comments following it for a detailed presentation of this result.
\end{itemize}


\section{Results of this work}
We begin the description of our results by noting that the Cauchy problem \eqref{eq:main}--\eqref{ini:L1Linf} has 
a local-in-time solution for a broad class of kernels
 satisfying $\nabla K \in L^1(\mathbb{R}^d) + L^\infty(\mathbb{R}^d)$ (see Proposition \ref{prop:local:exist} and Remark \ref{rem:L1Linfty}, below). We discuss this in Section~\ref{sec:preliminary}, where we also show that this solution is nonnegative, conserves mass 
\begin{equation}\label{mass:intro}
    M \equiv \int_{\mathbb{R}^d} u(x,t) \, \mathrm{d}x = \int_{\mathbb{R}^d} u_0(x) \, \mathrm{d}x \quad \text{for all } t\geq 0,
\end{equation}
and possesses sufficient regularity for the subsequent calculations. In fact, due to additional assumptions we impose on the kernel $K = K(x)$, these solutions exist for all $t \ge 0$ by a standard continuation argument combined with the $L^p$-estimates obtained below.

In Section \ref{sec:Lp:decay}, we estimate the $L^p$-norms of the solutions. We consider two types of kernels: either regular kernels in Theorem \ref{thm:regular:decay} or so-called repulsive kernels in Theorem~\ref{thm:decay:repulsive}. In both cases, we show that the $L^p$-norms of the solutions to problem \eqref{eq:main} with these kernels decay in time at the same rate as the solution 
\eqref{heat semigroup} to the Cauchy problem for the heat equation. These decay estimates play a crucial role in our study of the self-similar asymptotics of solutions to the Cauchy problem \eqref{eq:main}.


\subsection{Heat kernel asymptotics for regular kernels}
Our first two results concern the asymptotic simplification of solutions to the nonlinear Cauchy problem \eqref{eq:main} and their large-time behavior as characterized by the heat kernel. We begin by analyzing this model in the context of regular kernels.

\begin{theorem} \label{thm:regular:asymptotics:intro}
Let $d \ge 1$ and $M>0$. Assume that the interaction kernel $K$ satisfies 
\begin{equation}
    \nabla K(-x) = -\nabla K(x) \qquad \text{for all } x \in \mathbb{R}^d,
\end{equation}
and 
\begin{equation}
    \sup_{x \in \mathbb{R}^d} |x| \, |\nabla K(x)| < \infty.
\end{equation}
In addition, assume that one of the following conditions holds:
\begin{itemize}
    \item if $d = 1$, then $\nabla K \in L^1(\R^d)$.
    \item if $d \ge 2$, then $K \in W^{1,\infty}(\R^d),\quad \nabla K \in L^d(\R^d)$.
\end{itemize}
Let the initial datum $u_0 \in L^1(\mathbb{R}^d)\cap L^\infty(\mathbb{R}^d)$ be nonnegative and satisfy
\begin{equation*}
    \int_{\mathbb{R}^d} u_0(x) \, \mathrm{d}x = M, \quad
    \int_{\mathbb{R}^d} |x|^2 u_0(x) \, \mathrm{d}x < \infty, \quad \text{and} \quad
    \int_{\mathbb{R}^d} u_0(x)\log u_0(x) \, \mathrm{d}x < \infty.
\end{equation*}
Then, for every $p \in [1,\infty]$, the solution $u(x,t)$ to problem \eqref{eq:main} with mass $M=\|u_0\|_1$ satisfies the following large-time asymptotic behavior:
\begin{equation}
    t^{\frac{d}{2}\left(1-\frac{1}{p}\right)} \|u(t) - MG(t)\|_p \to 0 \qquad \text{as } t \to \infty.
\end{equation}
\end{theorem}

The self-similar heat kernel asymptotics of the aggregation-diffusion equation \eqref{eq:main} with regular kernels have been extensively studied \cite{KS10-spikes, CCS12, Carrillo-similification}. Notably, the work of Carrillo \textit{et al.}~\cite{Carrillo-similification} is particularly significant, as their results do not require smallness assumptions on either the initial data or the interaction kernel. While the proof of \cite[Theorem 1.2]{Carrillo-similification} utilizes entropy methods and relies on the relatively restrictive assumption that $\Delta K \in L^{\frac{d}{2}}(\mathbb{R}^d)$, we provide a simpler argument in Theorem \ref{thm:regular:asymptotics:intro} to establish these asymptotics. 
Our argument combines free-energy and relative-entropy bounds with $L^p$ estimates and the Duhamel formula; it does not impose an integrability assumption on $\Delta K$.
Moreover, in contrast to the result of \cite[Theorem 2.3]{KS10-spikes} in the one-dimensional case, our assumptions do not impose a smallness condition on the initial datum or the kernel.

\subsection{Heat kernel asymptotics for repulsive kernels}
Our second result shows that the asymptotic simplification of solutions toward the heat kernel can also be achieved for certain singular kernels. These singular kernels can be decomposed into the sum of a ``repulsive'' part and a part that is small in the Marcinkiewicz space $L^{d,\infty}(\mathbb{R}^d)$. We begin with a definition of the repulsive kernels used in this work.

\begin{definition}\label{def:repulsive kernel}
A kernel $K=K(x)$ is called repulsive if $\Delta K\le 0$ in the sense of distributions.
\end{definition}

Note that this definition in fact corresponds to 
 \emph{superharmonicity} condition and gives the favorable
sign in the $L^p$ energy identity. It is not, for arbitrary nonradial
kernels, equivalent to the physical condition $x\cdot\nabla K(x)\le0$.
The shifted point-source examples below are valid under the stated
definition, although they need not be even. 

The Cauchy problem \eqref{eq:main} with a repulsive, moderately singular kernel also possesses solutions whose large-time behavior is described by the heat kernel.

\begin{theorem} \label{thm:heat:kernel:intro}
Let $d\geq 2$ and $M>0$.
Let $u = u(x,t)$ be a nonnegative solution to the Cauchy problem \eqref{eq:main} with initial datum $u_0 \in L^1(\mathbb{R}^d) \cap L^\infty(\mathbb{R}^d)$, with  $M = \|u_0\|_1$.
Suppose that the kernel $K$ admits a decomposition $K = K_1 + K_2$ with the following properties 
\begin{itemize}
    \item $K_1$ is repulsive, i.e. $\Delta K_1 \le 0$ in the sense of distributions; moreover,  
    either  $\nabla K_1 \in L^{r,\infty}(\mathbb{R}^d)$ for some $r \in (1, d)$ or  $\nabla K_1 \in L^{d}(\mathbb{R}^d)$;
    \item $\nabla K_2 \in L^{d}(\mathbb{R}^d)$.
\end{itemize}
There exists a constant $D_d>0$, depending only on $d$, such that,
if $\|\nabla K_2\|_{d}\,M\le D_d$, then for every $p\in[1,\infty]$, 
\begin{equation}
t^{\frac{d}{2}\left(1-\frac{1}{p}\right)} \|u(t) - M G(t)\|_p \to 0 \qquad \text{as } t \to \infty.
\end{equation}
\end{theorem}

In other words, Theorem \ref{thm:heat:kernel:intro} states that solutions to the Cauchy problem \eqref{eq:main}--\eqref{ini:L1Linf} with a singular (but not too singular) repulsive kernel perturbed by a small regular part (from $L^d(\mathbb{R}^d)$) 
have Gaussian asymptotics as $t\to\infty$.

\begin{rem} \label{examples}
The assumptions in Theorem \ref{thm:heat:kernel:intro} cover several standard singular and nonsingular interaction kernels. Below, we list some examples satisfying these assumptions.

\begin{itemize}
\item \emph{Repulsive Riesz kernels.}
Let $d\ge 3$ and
\begin{equation}\label{repulsive Riesz}
    K_1(x)=A|x|^{-\alpha},
    \qquad
    A>0,\qquad 0<\alpha<d-2.
\end{equation}
Then
\[
    \Delta K_1(x)
    =
    -A\alpha(d-2-\alpha)|x|^{-\alpha-2}
    \le 0
\]
in the sense of distributions. Since
$|\nabla K_1(x)| \le A\alpha |x|^{-\alpha-1}$,
we immediately obtain 
$\nabla K_1 \in L^{\frac{d}{\alpha+1},\infty}(\mathbb{R}^d)$
with $1 < {d}/(\alpha+1) < d$.
The endpoint $\alpha=d-2$ gives the Newtonian kernel
\(
    K_1(x)=A|x|^{2-d}.
\)
In this case,
\[
    \Delta K_1=-c_dA\delta_0\le 0,
    \qquad
    \nabla K_1 \in L^{\frac{d}{d-1},\infty}(\mathbb{R}^d).
\]

\item \emph{Potentials generated by nonnegative measures.}
Let $d\ge 3$, let $\Phi_d$ be the fundamental solution of
$-\Delta$, and let $\mu$ be a finite nonnegative Radon measure.
Then $K_1=\Phi_d*\mu$ satisfies
\[
    \Delta K_1=-\mu\le 0.
\]
The weak-type estimate for the Riesz potential gives
$\nabla K_1 \in L^{\frac{d}{d-1},\infty}(\mathbb{R}^d)$.
For example, one may take
\[
    K_1(x)=\sum_{j=1}^N a_j\Phi_d(x-x_j),
    \qquad a_j\ge 0.
\]
This produces a repulsive kernel with finitely many point singularities.

\item \emph{A smooth repulsive kernel.}
For $d\ge 3$, consider
\[
    K_1(x)
    =
    A(1+|x|^2)^{-\frac{d-2}{2}},
    \qquad A>0.
\]
A direct calculation gives
$$
\Delta K_1(x)=-A d(d-2)(1+|x|^2)^{-(d+2)/2}\leq 0.
$$
Furthermore,
\[
    |\nabla K_1(x)|
    =
    A(d-2)|x|(1+|x|^2)^{-\frac{d}{2}},
\]
so that $\nabla K_1\in L^d(\mathbb{R}^d)$.
\end{itemize}
\end{rem}


\subsection{Critical repulsive case and self-similar solutions}
Finally, we consider the Cauchy problem \eqref{eq:main} in $\mathbb{R}^d$ for $d \ge 2$ with the kernel
obtained from the normalized limit
\[
  \lim_{\alpha\downarrow0}\frac{|x|^{-\alpha}-1}{\alpha}
  =-\log|x|,\qquad x\ne0,
\]
of the repulsive Riesz potential \eqref{repulsive Riesz}:
\begin{equation}
    K = -\log |x| \quad \text{satisfying} \quad \nabla K(x) = -\frac{x}{|x|^2} \in L^{d,\infty}(\mathbb{R}^d) \setminus L^d(\mathbb{R}^d).
\end{equation}
This kernel is excluded from the assumptions of Theorem~\ref{thm:heat:kernel:intro}, and the large-time behavior of solutions is essentially different. Here, in contrast to the subcritical case, the interaction term remains present at the diffusive scale, and the large-time asymptotic profile is expected to be a nonlinear self-similar solution rather than the heat kernel.

The gradient $\nabla K$ is homogeneous of degree $-1$. Consequently, a direct calculation (see also Remark~\ref{rem:scaling}) shows that equation \eqref{eq:main} is invariant under the parabolic rescaling
\begin{align}\label{def:rescaled-solution:0}
u_\lambda(x,t) := \lambda^d u(\lambda x, \lambda^2 t), \qquad x \in \mathbb{R}^d, \ \ t > 0,
\end{align}
for every $\lambda > 0$. That is, if $u$ is a solution, then $u_\lambda$ is a solution as well. As the third main result of this work, we show that equation \eqref{eq:main} admits a family of self-similar solutions -- that is, solutions invariant under \eqref{def:rescaled-solution:0} -- which govern the large-time behavior of general solutions.


\begin{theorem} \label{main thm:self-similar}
Let $d\geq 2$ and $M>0$. Assume that $\int_{\R^d}u_0(x)\,\d x=M$.
Let $u = u(x,t)$ be a non-negative solution to the Cauchy problem \eqref{eq:main} with kernel $K = -\log |x|$ and non-negative initial datum $u_0 \in L^1(\mathbb{R}^d) \cap L^\infty(\mathbb{R}^d)$ satisfying the finite second moment condition
$
\mu_2(0) := \int_{\mathbb{R}^d} |x|^2 u_0(x) \, \mathrm{d}x < \infty.
$
Then, for every $p \in [1, \infty]$, we have
\begin{align*}
t^{\frac{d}{2}\left(1-\frac{1}{p}\right)} \|u(t) - U_M(t)\|_{p} \to 0 \qquad \text{as } t \to \infty,
\end{align*}
where $U_M(x, t)$ is the unique self-similar solution of the equation in \eqref{eq:main} of the following form
\[
U_M(x,t) = t^{-\frac{d}{2}} \Phi_M\left(\frac{x}{\sqrt{t}}\right),
\]
where the profile $\Phi_M$ is a solution with mass $M$ to the following elliptic equation:
\begin{align} \label{main thm: stationary equation for W:0}
\Delta \Phi + \frac{1}{2} \nabla \cdot (y \Phi) + \nabla \cdot (\Phi \nabla K \ast \Phi) = 0 \quad \text{in } \mathbb{R}^d.
\end{align}
\end{theorem}

\begin{rem}
Theorem~\ref{main thm:self-similar} stands in contrast to the classical two-dimensional Keller--Segel system with logarithmic interaction, see {\it e.g.} \cite{MR3925816,BlanchetDolbeaultPerthame2006,BilerEtAl2006}. 
In the classical system, the logarithmic kernel is mass-critical, and the dynamics are strongly governed by critical mass thresholds, such as the $8\pi$ threshold. Consequently, the existence and asymptotic behavior of self-similar profiles depend sensitively on the total mass. In contrast, the logarithmic interaction in our setting has the {\it opposite, repulsive sign}. Although the kernel remains critical with respect to parabolic scaling, no mass threshold arises in Theorem~\ref{main thm:self-similar}.
\end{rem}

\begin{rem}
The proof of Theorem~\ref{main thm:self-similar} consists of three main steps. First, using scale-invariant estimates and a finite second moment, we show that the rescaled orbit is relatively compact in \(L^p(\mathbb{R}^d)\). 
Second, combining the Lyapunov functional in self-similar variables with the LaSalle invariance principle, we deduce that every element of the \(\omega\)-limit set is a weak stationary solution to the rescaled equation. 
Third, we show that this stationary problem admits at most one solution in the appropriate function class.
This uniqueness argument also departs from the approach used 
to study self-similar solutions 
for the attractive Keller--Segel system {\it e.g.} in the papers \cite{NaitoSuzuki04,Naito06,Nagai11}. 
The key ingredient is Lemma~\ref{lem:fourier inequality}, which establishes that the logarithmic interaction energy has a favorable sign on zero-mass perturbations. This sign property, combined with the strict monotonicity of the logarithm, yields the unique self-similar solution. Thus, uniqueness stems directly from the repulsive energy structure.
\end{rem}


\section{Preliminary properties of solutions}	\label{sec:preliminary}

First, we establish the local-in-time existence of solutions to problem \eqref{eq:main} for a large class of kernels $\nabla K$ and address fundamental properties including nonnegativity, mass conservation, and parabolic regularity. While these results are largely standard, they are included for the sake of completeness. Furthermore, the global existence of solutions to model \eqref{eq:main} follows from the \textit{a priori} estimates derived in the following sections.

\begin{prop}[Local-in-time solutions]\label{prop:local:exist}
	Let
	\begin{equation}\label{nablaK:ass}
	\nabla K = k_1 + k_2, \quad \text{where} \quad k_1 \in L^1(\R^d)\quad \text{and} \quad k_2 \in L^\infty(\R^d).
	\end{equation}
	Assume that $r \in (\max\{d, 2\}, \infty)$. 
	Then, for every $u_0 \in L^1(\R^d) \cap L^r(\R^d)$, there exists $T>0$ such that problem \eqref{eq:main} admits a unique mild solution in the space $C([0,T];L^1(\R^d)\cap L^r(\R^d))$.
\end{prop}

\begin{rem} \label{rem:L1Linfty}
	Assumption \eqref{nablaK:ass} encompasses a broad class of kernels, specifically those of the form:
	\[
	K(x) = 
	\begin{cases} 
	|x|^{-k}, & k \in (0, d-1), \\ 
	\log|x|, & k=0.
	\end{cases}
	\]
	These kernels exhibit the asymptotic gradient behavior $|\nabla K(x)| \sim |x|^{-(k+1)}$, which satisfies the requirements of Proposition \ref{prop:local:exist} via a natural decomposition 
    $$\nabla K = \nabla K \chi_{\{|x| \le 1\}} + \nabla K \chi_{\{|x| > 1\}},$$ 
    where the singularity at the origin is integrable since $k+1 < d$ and the tail is bounded.
	
	More generally, the decomposition property \eqref{nablaK:ass} is satisfied by every function $f$ in the weak Lebesgue space $L^{p,\infty}(\mathbb{R}^d)$ (also known as the Marcinkiewicz space) for $1 < p < \infty$. Indeed, given any threshold $\lambda > 0$, we may decompose $f$ into $f = f_1 + f_2$, where $f_1 = f \chi_{\{|f| > \lambda\}}$ and $f_2 = f \chi_{\{|f| \le \lambda\}}$. It is immediate that $f_2 \in L^\infty(\mathbb{R}^d)$ with $\|f_2\|_{\infty} \le \lambda$. To show that $f_1 \in L^1(\mathbb{R}^d)$, we employ the layer-cake representation alongside the weak $L^p$ estimate $|\{|f| > t\}| \le \|f\|_{L^{p,\infty}}^p t^{-p}$:
	\begin{align*}
		\|f_1\|_{1} &= \int_0^\infty \big|\{x \in \mathbb{R}^d : |f_1(x)| > t\}\big|  \, \mathrm{d}t \\
		&= \int_0^\lambda \big|\{|f| > \lambda\}\big|  \, \mathrm{d}t + \int_\lambda^\infty \big|\{|f| > t\}\big|  \, \mathrm{d}t \\
		&\le \lambda \frac{\|f\|_{L^{p,\infty}}^p}{\lambda^p} + \|f\|_{L^{p,\infty}}^p \int_\lambda^\infty t^{-p}  \, \mathrm{d}t < \infty.
	\end{align*}
\end{rem}

\begin{proof}[Proof of Proposition \ref{prop:local:exist}.]
We work in the Banach space
\[
X_T := C([0,T]; L^1(\R^d) \cap L^r(\R^d))
\]
endowed with the norm
\[
\|u\|_{X_T} := \sup_{0 \le t \le T} \left( \|u(t)\|_1 + \|u(t)\|_r \right).
\]
For $u,v \in X_T$, we introduce the bilinear form
\begin{equation}\label{bilinear form}
    B(u,v)(t) := \int_0^t \nabla \cdot e^{(t-s)\Delta} (u(\nabla K \ast v))(s) \, \mathrm{d}s, \qquad t \in [0,T].
\end{equation}
Then the mild formulation of the Cauchy problem for equation \eqref{eq:main} can be written as
\begin{equation}\label{duhamel}
    u(t) = e^{t\Delta}u_0 + B(u,u)(t), \qquad t \in [0,T].
\end{equation}
Since the heat semigroup estimate guarantees that $e^{t\Delta}u_0$ lies in $X_T$, the construction of a mild solution reduces to establishing the appropriate bounds for the bilinear form \eqref{bilinear form}.

\medskip
\noindent
\textit{Step 1. $L^1$-estimate.}
We apply the H\"older inequality and the Young inequality to the singular term involving $k_1$:
\begin{align*}
    \|u(k_1 \ast v)\|_1 \le \|u\|_2 \|k_1\|_1 \|v\|_2.
\end{align*}
Since $r > \max\{d, 2\} \ge 2$, interpolating between $L^1$ and $L^r$ yields
\[
    \|u\|_2 \le C\|u\|_1^\theta \|u\|_r^{1-\theta} \le C\|u\|_{X_T}
\]
for $\theta = \frac{r-2}{2(r-1)} \in [0,1]$. Consequently, combining these bounds with formula \eqref{bilinear form}, we arrive at
\begin{align*}
    \left\| \int_0^t \nabla \cdot e^{(t-s)\Delta} (u(k_1 \ast v))(s) \, \mathrm{d}s \right\|_1 &\le \int_0^t (t-s)^{-\frac{1}{2}} \|u(k_1 \ast v)(s)\|_1 \, \mathrm{d}s \\
    &\le C T^{\frac{1}{2}} \|k_1\|_1 \|u\|_{X_T} \|v\|_{X_T}.
\end{align*}
For the term involving $k_2$, we have
\begin{align*}
    \|u(k_2 \ast v)\|_1 &\le \|u\|_1 \|k_2\|_\infty \|v\|_1 \\
    &\le \|k_2\|_\infty \|u\|_{X_T} \|v\|_{X_T}.
\end{align*}
Thus, the $L^1$-norm of the $k_2$ contribution to the integral is similarly bounded by
\[
C T^{\frac{1}{2}} \|k_2\|_\infty \|u\|_{X_T} \|v\|_{X_T}.
\]

\medskip
\noindent
\textit{Step 2. $L^r$-estimate.}
Define $m$ by $\frac{1}{m} = \frac{2}{r}$. Since $r \ge 2$, we have $m \ge 1$. Using the H\"older inequality and the Young inequality, we estimate the nonlinear term:
\begin{align*}
    \|u(k_1 \ast v)\|_m &\le \|u\|_r \|k_1 \ast v\|_r \\
    &\le \|k_1\|_1 \|u\|_r \|v\|_r \\
    &\le \|k_1\|_1 \|u\|_{X_T} \|v\|_{X_T}.
\end{align*}
Setting $\beta_1 = \frac{1}{2} + \frac{d}{2}\left(\frac{1}{m} - \frac{1}{r}\right) = \frac{1}{2} + \frac{d}{2r}$, we note that our assumption $r > d$ ensures $\beta_1 < 1$. Hence, the heat semigroup estimate yields
\[
    \left\| \int_0^t \nabla \cdot e^{(t-s)\Delta} (u(k_1 \ast v))(s) \, \mathrm{d}s \right\|_r \le C T^{1-\beta_1} \|k_1\|_1 \|u\|_{X_T} \|v\|_{X_T}.
\]
For the contribution involving $k_2$, we estimate
\begin{align*}
    \|u(k_2 \ast v)\|_r &\le \|u\|_r \|k_2 \ast v\|_\infty \\
    &\le \|u\|_r \|k_2\|_\infty \|v\|_1 \\
    &\le \|k_2\|_\infty \|u\|_{X_T} \|v\|_{X_T}.
\end{align*}
In this case, the corresponding singular exponent is $\beta_2 = \frac{1}{2} < 1$, which yields the bound $C T^{\frac{1}{2}} \|k_2\|_\infty \|u\|_{X_T} \|v\|_{X_T}$.

\medskip
\noindent
\textit{Step 3. Existence.}
Combining the above estimates, we obtain
\[
    \|B(u,v)\|_{X_T} \le C\,\omega(T)\,\|u\|_{X_T}\|v\|_{X_T},
\]
where $\omega(T) \to 0$ as $T \downarrow 0$. Thus, for sufficiently small $T$, a local-in-time solution to equation \eqref{duhamel} is constructed via the Banach fixed-point theorem.

\medskip
\noindent
\textit{Step 4. Uniqueness.}
Let $u_1,u_2\in X_T$ have the same initial datum, and set
\[
 H(t)=\|u_1(t)-u_2(t)\|_1+\|u_1(t)-u_2(t)\|_r.
\]
The estimates in Steps 1 and 2, applied to
$B(u_1,u_1)-B(u_2,u_2)=B(u_1-u_2,u_1)+B(u_2,u_1-u_2)$, yield
\[
 H(t)\le C_T\int_0^t
 \bigl((t-s)^{-1/2}+(t-s)^{-\beta_1}\bigr)H(s)\,\d s,
 \qquad \beta_1=\frac12+\frac{d}{2r}<1.
\]
Here $C_T$ is finite because $u_1,u_2\in X_T$. The singular
Gr\"onwall inequality, or contraction on successive short intervals,
gives $H\equiv0$. This proves uniqueness in $X_T$.
\end{proof}

\begin{prop}\label{prop:properties}
    The local-in-time mild solution $u=u(x,t)$ constructed in Proposition~\ref{prop:local:exist} satisfies the following properties:
    \begin{enumerate}
        \item {Non-negativity:} $u(x,t) \ge 0$ for all $t \in [0,T]$ and $x \in \mathbb{R}^d$, provided $u_0 \ge 0$.
        \item {Mass conservation:} The total mass is preserved for all $t \in [0, T]$, namely,
        \begin{equation}\label{mass}
        M \equiv \int_{\mathbb{R}^d} u(x,t)  \, \mathrm{d}x = \int_{\mathbb{R}^d} u_0(x)  \, \mathrm{d}x.
        \end{equation}
        \item {Parabolic regularity:} The solution possesses classical regularity of the form:
        \begin{equation} \label{regularity}
        \begin{aligned}
            u&\in C\big([0,T];L^1(\mathbb R^d)\cap L^r(\mathbb R^d)\big),\\
            u&\in C\big((0,T];W^{2,p}(\mathbb R^d)\big)
                 \cap C^1\big((0,T];L^p(\mathbb R^d)\big)
            \end{aligned}
            \end{equation}
            for every $p > 1$.
    \end{enumerate}
\end{prop}

\begin{proof}
The proof of Proposition \ref{prop:properties} is standard; analogous arguments have been established in previous literature (see, e.g., \cite{LR09, LR10,KarchSuzuki2011,karch2026concentrationmasssolutionsaggregationdiffusion,Carrillo-similification}). Consequently, we only provide a brief outline of the key steps.

Mass conservation is obtained by integrating both sides of equation \eqref{duhamel} with respect to $x \in \mathbb{R}^d$ and applying the following relations for the heat semigroup 
$$
\int_{\R^d} e^{t\Delta}w \, \mathrm{d}x=\int_{\R^d} w \, \mathrm{d}x \quad\text{and}\quad  \int_{\R^d} \nabla e^{t\Delta}w \, \mathrm{d}x=0
$$
for all $w\in L^1(\R^d)$ and $t>0$.

The parabolic regularity \eqref{regularity} is also a classical result for such systems, see {\it e.g.} \cite{Carrillo-similification}. 
\end{proof}

\begin{prop}\label{prop:bound for moment}
Assume that a nonnegative initial condition $u_0 \in L^1(\R^d) \cap L^\infty(\R^d)$ satisfies 
$
\mu_2 (0) = \int_{\R^d} |x|^2 u_0 (x)  \, \mathrm{d}x < \infty.
$
 There exists a number $C_T>0$ such that  the local-in-time mild solution $u=u(x,t)$ constructed in Proposition~\ref{prop:local:exist} satisfies 
 \begin{equation}
     \mu_2 (t) \equiv \int_{\R^d} |x|^2 u (x,t)  \, \mathrm{d}x \leq C_T \quad \text{for all } t\in [0,T].
 \end{equation}
\end{prop}
\begin{proof}
The parabolic regularity \eqref{regularity} for $p > d/2$ implies that $u \in C((0,T]; L^\infty(\mathbb{R}^d))$ for every $\tau > 0$. Thus, using the decomposition $\nabla K = k_1 + k_2$ with $k_1 \in L^1(\mathbb{R}^d)$ and $k_2 \in L^\infty(\mathbb{R}^d)$, along with Young's inequality, in the same way as in the proof of Proposition~\ref{prop:local:exist}, we obtain
\begin{equation}
    \nabla K * u \in L^\infty(\tau, T; L^\infty(\mathbb{R}^d))
\end{equation}
for every $\tau > 0$. This regularity is sufficient to justify the standard cut-off argument for the second moment on $[\tau, T]$ (see \cite[Remark 2.3]{CPZ04} and \cite[Lemma 3.2]{KSu08}).
\end{proof}


\section{\texorpdfstring{$L^p$}{Lp}-decay of solutions}\label{sec:Lp:decay}

The first step in studying the large-time behavior of solutions to the Cauchy problem \eqref{eq:main}--\eqref{ini:L1Linf} is to obtain decay estimates of their $L^p$-norms. In this section, we establish decay estimates analogous to those for the heat semigroup for two classes of kernels:
arbitrary  regular kernels and
singular repulsive kernels.


\subsection{Decay in the case of regular kernels}

\begin{theorem} \label{thm:regular:decay}
Let $d \ge 1$ and $M>0$. Assume that the interaction kernel $K$ satisfies 
\begin{equation}
    \nabla K(-x) = -\nabla K(x) \qquad \text{for all } x \in \mathbb{R}^d,
\end{equation}
and 
\begin{equation}\label{ass:reg:K:decay}
    \sup_{x \in \mathbb{R}^d} |x| \, |\nabla K(x)| < \infty.
\end{equation}
In addition, assume that one of the following conditions holds:
\begin{itemize}
    \item if $d = 1$, then $\nabla K \in L^1(\R^d)$.
    \item if $d \ge 2$, then $K \in W^{1,\infty}(\R^d),\quad \nabla K \in L^d(\R^d)$.
\end{itemize}
Let the initial datum 
$u_0 \in L^1(\mathbb{R}^d) \cap L^\infty(\R^d)$  be nonnegative and satisfy
\begin{equation*}
    \int_{\mathbb{R}^d} u_0(x)  \, \mathrm{d}x = M, \quad
    \int_{\mathbb{R}^d} |x|^2 u_0(x)  \, \mathrm{d}x < \infty, \quad \text{and} \quad
    \int_{\mathbb{R}^d} u_0(x)\log u_0(x)  \, \mathrm{d}x < \infty.
\end{equation*}
Then, for every $p \in [1,\infty]$, the solution $u(x,t)$ to problem \eqref{eq:main} with mass $M=\|u_0\|_1$
satisfies the following decay estimate:
\begin{equation}
   \|u(t)\|_p \leq C(M,p,d,\nabla K,u_0)   t^{- \frac{d}{2}\left(1-\frac{1}{p}\right)}
    \qquad \text{for all } t>0.
\end{equation}
\end{theorem}



\begin{rem} 
   In the paper \cite{KS10-spikes}, the $L^p$-decay of solutions to model \eqref{eq:main}--\eqref{ini:L1Linf} in one spatial dimension was obtained for all kernels $K = K(x)$ such that $\nabla K \in L^1(\mathbb{R})$ and the quantity $\|\nabla K\|_1 \|u_0\|_1$ is sufficiently small. Note that Theorem \ref{thm:regular:decay} is also stated for $d = 1$ and allows us to remove the smallness assumption from \cite{KS10-spikes}, albeit under extra assumptions on the kernel and initial data, as stated in Theorem \ref{thm:regular:decay}.
\end{rem}

We proceed with the proof of Theorem \ref{thm:regular:decay} 
in several steps
by establishing some preliminary estimates. 

First, following \cite{Carrillo-similification}, we recall that equation \eqref{eq:main} admits the following Lyapunov functional:
\begin{equation}\label{entropy}
    E[u(t)] = \int_{\R^d} u\log u \, \mathrm{d}x + \frac{1}{2}\int_{\R^d} u (K * u) \, \mathrm{d}x.
\end{equation}
The dissipation of this functional is given by
\begin{equation}\label{entropy dyssipation}
    \frac{d}{dt}E[u(t)] = -\int_{\R^d} u|\nabla (\log u + K*u)|^2  \, \mathrm{d}x.
\end{equation}

The following estimate is a direct consequence of this Lyapunov structure. It can be obtained by directly adapting the argument from \cite{Carrillo-similification} to the case of a general mass $\|u_0\|_{L^1(\R^d)}=M$; therefore, we omit the proof.

\begin{lemma}[{\cite[Lemma 3.1]{Carrillo-similification}}]\label{lem:entropy_estimate}
Assume 
the regular-kernel hypotheses of Theorem~\ref{thm:regular:decay}; in particular, $K$ is even and bounded.
Let $u$ be a nonnegative solution to problem \eqref{eq:main}--\eqref{ini:L1Linf} satisfying
$
    E[u_0] < \infty.
$ 
Then, it holds that
\begin{equation*}
    E[u(t)] \leq -\frac{d\|u_0\|_{1}}{2}\log \left(c t + e^{-\frac{2}{d\|u_0\|_{1}}E[u_0]}\right),
\end{equation*}
where the constant $c > 0$ depends only on $\|K \|_{\infty}$, $\|u_0\|_{1}$, and the dimension $d$.
\end{lemma}

\begin{rem}
 Applying Lemma \ref{lem:entropy_estimate}, we immediately obtain the following bound, which demonstrates that $E[u(t)]$ escapes to $-\infty$ logarithmically in time:
\begin{equation}\label{eq:entropy:estimate}
\begin{split}
E[u(t)] &\leq -\dfrac{d\|u_0\|_{1}}{2} \log \left\{t\left(c + \dfrac{e^{-\frac{2}{d\|u_0\|_{1}}E[u_0]}}{t}\right)\right\}\\
&= -\dfrac{d\|u_0\|_{1}}{2}\log t -\dfrac{d\|u_0\|_{1}}{2}\log \left(c + \dfrac{e^{-\frac{2}{d\|u_0\|_{1}}E[u_0]}}{t}\right)\\
&\leq -\dfrac{d\|u_0\|_{1}}{2}\log t -\dfrac{d\|u_0\|_{1}}{2}\log c.
\end{split}
\end{equation}
\end{rem}

\begin{lemma}\label{lem:ulogu:G}
For the heat kernel \eqref{heat kernel}
and for  arbitrary nonnegative  $f\in L^1(\R^d)$ such that $f\log f \in L^1(\R^d)$ and $\int |x|^2f(x)\,\d x<\infty$, the following identity holds 
\begin{align*}
\int_{\R^d}f \log \dfrac{f}{G}  \, \mathrm{d}x = \int_{\R^d} f\log f  \, \mathrm{d}x + \dfrac{d\|f\|_{1}}{2}\log (4\pi t) + \dfrac{\int_{\R^d}|x|^2 f(x)  \, \mathrm{d}x}{4t}.
\end{align*}
\end{lemma}
\begin{proof}
    Substituting the explicit Gaussian formula gives the identity.
\end{proof}

Lemma \ref{lem:ulogu:G} reduces the study of the relative entropy with respect to the heat kernel to the control of the entropy and the second moment. 
We next prove that $\mu_2(t)$, which is denoted by
\begin{align*}
    \mu_2(t) := \int_{\R^d}|x|^2 u(x,t)  \, \mathrm{d}x,
\end{align*}
grows at most linearly under suitable assumptions on $K$.

\begin{lemma}\label{lem:moment} Let $K$ and $u_0$ satisfy the assumptions of Theorem~\ref{thm:regular:decay}, and let $u=u(x,t)$ be a nonnegative solution to \eqref{eq:main} in $\R^d \times (0,\infty)$.
Then the second moment $\mu_2(t)$ satisfies
\begin{align}\label{eq:moment}
    \mu_2(t) \leq \mu_2(0) + (2dM + \sup_{x \in \R^d} |x||\nabla K(x)|M^2)t\quad\text{for all } t > 0.
\end{align}
\end{lemma}
\begin{proof}
We remark that the following calculation is justified owing to Proposition~\ref{prop:bound for moment}.
Differentiating $\mu_2(t)$ and using \eqref{eq:main}, we obtain
\begin{align*}
    \frac{d}{dt}\mu_2(t)
    &= \frac{d}{dt}\int_{\R^d}|x|^2u(x,t) \, \mathrm{d}x \\
    &= \int_{\R^d}|x|^2 u_t(x,t) \, \mathrm{d}x \\
    &= 2dM - 2\int_{\R^d}u(x,t)\,x\cdot (\nabla K * u)(x,t) \, \mathrm{d}x.
\end{align*}
Employing the oddness of $\nabla K$, we rewrite the interaction term as
\begin{align*}
    -2\int_{\R^d}u(x,t)\,x\cdot (\nabla K * u)(x,t) \, \mathrm{d}x
    = -\int_{\R^d}\int_{\R^d}(x-y)\cdot \nabla K(x-y)\,u(x,t)u(y,t) \, \mathrm{d}y \, \mathrm{d}x.
\end{align*}
Therefore, by the assumption
\[
\sup_{z\in\R^d}|z|\,|\nabla K(z)|<\infty,
\]
we infer that
\begin{align*}
    \frac{d}{dt}\mu_2(t)
    &\le 2dM + \sup_{z\in\R^d}|z|\,|\nabla K(z)|
    \int_{\R^d}\int_{\R^d}u(x,t)u(y,t) \, \mathrm{d}y \, \mathrm{d}x \\
    &= 2dM + \sup_{z\in\R^d}|z|\,|\nabla K(z)|\,M^2.
\end{align*}
Integrating this differential inequality over $(0,t)$, we conclude
\[
\mu_2(t)\le \mu_2(0) + Ct
\qquad \text{for all } t>0,
\]
where
\[
C:=2dM + \sup_{z\in\R^d}|z|\,|\nabla K(z)|\,M^2.
\]
This proves that $\mu_2(t)$ grows at most linearly in time.
\end{proof}

Next, we establish a uniform bound on the relative entropy with respect to the heat kernel, which follows from the Lyapunov structure combined with control over the second moment.

\begin{lemma}\label{lem:relative entropy bound}
Assume the kernel hypotheses of Theorem~\ref{thm:regular:decay}.
Let $u(x,t)$ be a nonnegative solution to problem \eqref{eq:main} with initial datum $u_0 \in L^1(\R^d)\cap L^\infty(\R^d)$ satisfying
\[
u_0 \ge 0, \quad \int_{\R^d} |x|^2 u_0(x) \, \mathrm{d}x < \infty, \quad E[u_0] < \infty.
\]
Then it holds that
\begin{align}\label{eq:relative entropy bound}
\sup_{t\ge t_0} \int_{\R^d} u(x,t)\log \frac{u(x,t)}{G(x,t)} \, \mathrm{d}x < \infty
\end{align}
for some $t_0 > 0$.
\end{lemma}
\begin{proof}
We begin by the identity obtained from formula \eqref{entropy} and Lemma \ref{lem:ulogu:G}:
\begin{align*}
\int_{\R^d} u \log \frac{u}{G} \, \mathrm{d}x
= E[u(t)] - \frac{1}{2}\int_{\R^d} uK*u \, \mathrm{d}x 
+ \frac{dM}{2}\log (4\pi t) 
+ \frac{\mu_2(t)}{4t}.
\end{align*}
It follows from the estimate \eqref{eq:entropy:estimate} that
\[
E[u(t)] \le -\frac{dM}{2}\log t + C.
\]
Moreover, since $K \in L^\infty(\R^d)$, we apply the Young inequality to obtain
\[
-\frac{1}{2}\int_{\R^d} uK*u \, \mathrm{d}x 
\le \frac{1}{2}\|K\|_{\infty}M^2.
\]
Combining the above estimates with the linear growth estimate \eqref{eq:moment} for the second moment $\mu_2(t)$, we conclude that
\[
\int_{\R^d} u \log \frac{u}{G} \, \mathrm{d}x
\le C + \frac{\mu_2(0)}{4t},
\]
which yields the desired boundedness.
\end{proof}
The next elementary estimate allows us to control the mass of $u$ on measurable sets in terms of the relative entropy with respect to the heat kernel.
\begin{lemma}\label{measure theory}
Let $E \subset \R^d$ be a nonempty measurable set of nonzero measure. Then it follows that
\begin{align*}
\left(\int_E u(x,t) \, \mathrm{d}x\right)
\log \frac{\int_E u(x,t) \, \mathrm{d}x}{\int_E G(x,t) \, \mathrm{d}x}
\le
\int_{\R^d} u\log \frac{u}{G} \, \mathrm{d}x + \frac{1}{e}.
\end{align*}
\end{lemma}

\begin{proof}
Let us denote $a(t)$ and $b(t)$ by
\begin{align*}
a(t) = \int_E u(x,t)  \, \mathrm{d}x,\quad b(t) = \int_E G(x,t)  \, \mathrm{d}x.
\end{align*}
Then the Jensen inequality gives 
\begin{align}
\int_E u\log \dfrac{u}{G}  \, \mathrm{d}x &= \int_E G\left(\dfrac{u}{G}\log \dfrac{u}{G}\right)  \, \mathrm{d}x\notag\\
&\ge b(t) \left(\int_E \dfrac{u}{b(t)}  \, \mathrm{d}x \right)\log \left(\int_E \dfrac{u}{b(t)}  \, \mathrm{d}x\right)\notag\\
&= a(t) \log \dfrac{a(t)}{b(t)}.\label{eq1:measure theory}
\end{align}
If $E = \R^d$, then our claim is immediately obtained. If $E \subsetneqq \R^d$, then it follows from $G(x,t) > 0$ for all $(x,t) \in \R^d \times (0,\infty)$ that
\begin{align*}
b(t) = \int_E G(x,t) < 1\quad \text{for all}\  t > 0.
\end{align*}
Similarly, we obtain
\begin{align}\label{eq2:measure theory}
\int_{E^c} u\log \dfrac{u}{G}  \, \mathrm{d}x &\ge  (M - a(t)) \log \dfrac{M-a(t)}{1-b(t)}.
\end{align}
Therefore combining \eqref{eq1:measure theory} and \eqref{eq2:measure theory},
we get
\begin{align*}
\int_{\R^d} u\log \dfrac{u}{G}  \, \mathrm{d}x \ge a(t) \log \dfrac{a(t)}{b(t)}  + (\|u_0\|_{1} - a(t))\log \dfrac{\|u_0\|_{1} - a(t)}{1-b(t)}.
\end{align*}
Since $a(t) \leq M$ and $b(t) < 1$, we can calculate as follows
\begin{align*}
&(M - a(t) )\log \dfrac{M - a(t)}{1-b(t)}\\
&= (M - a(t)) \log (M - a(t)) -(M - a(t))\log (1-b(t))\\
&\ge -\dfrac{1}{e}.
\end{align*}
Consequently we conclude that
\begin{align*}
\int_{\R^d} u\log \dfrac{u}{G}  \, \mathrm{d}x \ge a(t) \log \dfrac{a(t)}{b(t)} - \dfrac{1}{e},
\end{align*}
which implies our claim.
\end{proof}

It is well known that the $L^p$ norms of solutions to problem \eqref{eq:main} with regular kernels remain uniformly bounded in time. 
For completeness, we recall a  result from \cite{BBKL21}, see also \cite{KS10-spikes} for estimates for the one-dimensional model.

\begin{lemma}[{\cite[Lemma 4.1]{BBKL21}}] \label{lem:bound}
    Let $u$ be a nonnegative solution to problem \eqref{eq:main} with the interaction kernel satisfying $\nabla K \in L^\infty(\mathbb{R}^d)$ and with an initial condition satisfying
\begin{equation}
    u_0 \in L^1(\mathbb{R}^d) \cap L^\infty(\mathbb{R}^d), \quad u_0 \geq 0, \quad M = \int_{\mathbb{R}^d} u_0(x) \, \mathrm{d}x > 0.
\end{equation}
For each $p \in [1, \infty)$ there exists a constant $C_p=C(p, d, M,\|u_0\|_p,\|\nabla K\|_\infty) > 0$ such that
\[
    \|u(t)\|_p \leq C_p \quad\text{for all } t \geq 0.
\]
\end{lemma}

The following lemma provides the first step toward understanding the large-time behavior of solutions to problem \eqref{eq:main}.

\begin{lemma}\label{lem:decay}
Assume the hypotheses of Theorem~\ref{thm:regular:decay}.
Let $u$ be a global-in-time solution to problem \eqref{eq:main} such that
\[
    \sup_{t > 0}\|u(t)\|_{r} < \infty \quad \text{for some } r \in (1,\infty).
\]
Then it holds that for all $p \in (1,r)$,
\[
    \lim_{t \to \infty}\|u(t)\|_{p} = 0.
\]
\end{lemma}

\begin{proof}
It suffices to consider $t$ large enough such that
\[
    \max\{\|u_0\|_{1}, e\} < \log t.
\]
Define the set
\[
    A(t) := \{x \in \R^d : u(x,t) \ge G(x,t) \log t \}.
\]
We first show that $A(t) \neq \R^d$. Indeed, if $A(t) = \R^d$ for large $t$, then integrating over $\R^d$ yields
\[
    \|u_0\|_{1} = \int_{\R^d}u(x,t)  \, \mathrm{d}x \ge \log t \int_{\R^d} G(x,t)  \, \mathrm{d}x = \log t,
\]
which contradicts our choice of $t$. 
Assume for now that $A(t)$ is nonempty for large $t$. By the definition of $A(t)$, we have
\[
    \int_{A(t)} u(x,t)  \, \mathrm{d}x \ge \log t \int_{A(t)} G(x,t) \, \mathrm{d}x.
\]
This immediately implies
\[
    \left(\int_{A(t)} u(x,t)  \, \mathrm{d}x\right) \log \left(\dfrac{\int_{A(t)} u(x,t)  \, \mathrm{d}x}{\int_{A(t)} G(x,t)  \, \mathrm{d}x}\right) \ge \left(\int_{A(t)} u(x,t)  \, \mathrm{d}x\right) \log \log t.
\]
Since $e < \log t$, we know that $\log \log t > 0$. Applying Lemma~\ref{measure theory} with $E = A(t)$, we obtain
\[
    \int_{A(t)} u(x,t)  \, \mathrm{d}x \leq \dfrac{\int_{\R^d} u\log \left(\frac{u}{G}\right)  \, \mathrm{d}x + \frac{1}{e}}{\log \log t}.
\]
In view of the relative entropy bound \eqref{eq:relative entropy bound}, this implies
\[
    \lim_{t\to\infty} \int_{A(t)} u(x,t) \, \mathrm{d}x = 0.
\]
Therefore, for $p \in (1,r)$, splitting the integral and applying the H\"older inequality (with $\alpha = \frac{r-p}{p(r-1)}$) yields
\begin{align*}
    \int_{\R^d} |u(x,t)|^p  \, \mathrm{d}x &= \int_{A(t)}|u(x,t)|^p  \, \mathrm{d}x + \int_{A(t)^c}|u(x,t)|^p  \, \mathrm{d}x\\
    &\leq \left(\int_{A(t)}|u(x,t)|  \, \mathrm{d}x\right)^{p\alpha} \left(\int_{A(t)}|u(x,t)|^r  \, \mathrm{d}x\right)^\frac{p(1-\alpha)}{r}\\
    &\quad + (\log t)^{p-1}\int_{A(t)^c}u(x,t) G(x,t)^{p-1}  \, \mathrm{d}x.
\end{align*}
Moreover, standard properties of the heat kernel $G$ allow us to estimate the second term as follows:
\begin{align*}
    \int_{\R^d} |u(x,t)|^p  \, \mathrm{d}x &\leq \left(\int_{A(t)}|u(x,t)|  \, \mathrm{d}x\right)^{p\alpha} \left(\int_{A(t)}|u(x,t)|^r  \, \mathrm{d}x\right)^\frac{p(1-\alpha)}{r}\\
    &\quad + (\log t)^{p-1}\|G(t)\|_{\infty}^{p-1} \int_{A(t)^c}u(x,t)  \, \mathrm{d}x\\
    &\leq \left(\int_{A(t)}|u(x,t)|  \, \mathrm{d}x\right)^{p\alpha} \left(\int_{A(t)}|u(x,t)|^r  \, \mathrm{d}x\right)^\frac{p(1-\alpha)}{r}\\
    &\quad + C^{p-1} (\log t)^{p-1} t^{-\frac{d(p-1)}{2}}\|u_0\|_{1}.
\end{align*}
Since the $L^r$ norm is bounded by assumption and $t^{-\frac{d(p-1)}{2}}(\log t)^{p-1} \to 0$ as $t \to \infty$, we conclude that for any $p \in (1,r)$,
\[
    \lim_{t \to \infty}\|u(t)\|_{p} = 0.
\]
Finally, if $A(t)$ is empty for large $t$, the claim easily follows since the estimate on $A(t)^c$ directly applies to the whole space $\R^d$. 
\end{proof}

Finally, we refine Lemma \ref{lem:decay} to obtain optimal decay rates, thereby completing the proof of Theorem \ref{thm:regular:decay}.

\begin{proof}[Proof of Theorem \ref{thm:regular:decay}]
Detailed calculations for $d=1$ can be found in \cite[Thm.~2.5]{KS10-spikes}; 
we therefore assume $d\ge 2$ in what follows.

By Lemma \ref{lem:bound} and Lemma \ref{lem:decay}, we have $\lim_{t \to \infty} \|u(t)\|_p = 0$ for each $p\in (1,\infty)$. 
First, we  establish the optimal decay estimate for every \(p\in [2,\infty)\). 
Set
\[
w=u^{\frac{p}{2}},
\qquad
Y_p(t)=\|u(t)\|_p^p=\|w(t)\|_2^2.
\]
Testing equation \eqref{eq:main} with \(u^{p-1}\) gives
\begin{equation}\label{Lp-energy-identity}
	\begin{split}
		\frac{1}{p}\frac{d}{dt}Y_p(t)
		={}&-\frac{4(p-1)}{p^2}\|\nabla w(t)\|_2^2\\
		&-\frac{2(p-1)}{p}
		\int_{\mathbb R^d}
		w(t)\nabla w(t)\cdot (\nabla K*u(t)) \, \mathrm{d}x.
	\end{split}
\end{equation}
We decompose
\[
L^d(\R^d)\ni \nabla K=(\nabla K-\Psi)+\Psi,
\]
where \(\Psi\in C_c^\infty(\mathbb R^d;\mathbb R^d)\) is chosen so that
\[
\|\nabla K-\Psi\|_d\le \delta.
\]

We first consider \(d\ge 3\). By the H\"older inequality, the Sobolev inequality applied to \(w=u^{p/2}\), and Young's convolution inequality, we obtain
\begin{equation}\label{Lp-first-term-dge3}
	\begin{split}
		\left|
		\int_{\mathbb R^d}
		w\nabla w\cdot ((\nabla K-\Psi)*u) \, \mathrm{d}x
		\right|
		&\le
		\|w\|_{\frac{2d}{d-2}}
		\|\nabla w\|_2
		\|(\nabla K-\Psi)*u\|_d\\
		&\le
		C_S\|\nabla w\|_2^2
		\|\nabla K-\Psi\|_d\|u\|_1\\
		&\le
		C_S\delta M\|\nabla w\|_2^2.
	\end{split}
\end{equation}
Similarly,
\begin{equation}\label{Lp-second-term-dge3}
	\begin{split}
		\left|
		\int_{\mathbb R^d}
		w\nabla w\cdot (\Psi*u) \, \mathrm{d}x
		\right|
		&\le
		\|w\|_{\frac{2d}{d-2}}
		\|\nabla w\|_2
		\|\Psi*u\|_d\\
		&\le
		C_S\|\Psi\|_1\|u(t)\|_d
		\|\nabla w\|_2^2.
	\end{split}
\end{equation}
Consequently,
\begin{equation}\label{Lp-differential-dge3}
	\begin{split}
		\frac{1}{p}\frac{d}{dt}Y_p(t)
		\leq
		\Bigg[
		-\frac{4(p-1)}{p^2}
		+\frac{2(p-1)}{p}C_S
		\Big(
		\delta M+\|\Psi\|_1\|u(t)\|_d
		\Big)
		\Bigg]
		\|\nabla w(t)\|_2^2.
	\end{split}
\end{equation}
We first choose \(\delta>0\) sufficiently small and then use
\[
\lim_{t\to\infty}\|u(t)\|_d=0
\]
to find \(T_p>0\) such that
\[
C_S\Big(
\delta M+\|\Psi\|_1\|u(t)\|_d
\Big)\le \frac1p,
\qquad t\ge T_p.
\]
It follows from \eqref{Lp-differential-dge3} that
\begin{equation}\label{Lp-dissipation}
	\frac{d}{dt}Y_p(t)
	\le
	-c_p\|\nabla u^{\frac{p}{2}}(t)\|_2^2,
	\qquad t\ge T_p,
\end{equation}
where, for example, one may take
\[
c_p=\frac{2(p-1)}p>0.
\]

We now establish the analogous estimate when \(d=2\). We use the following consequences of the two-dimensional Gagliardo--Nirenberg inequality:
\begin{equation}\label{GN-d2-mass-p}
	\|u^\frac{p}{2}\|_4
	\le
	C_p M^{\frac{1}{4}}
	\|\nabla u^{\frac{p}{2}}\|_2^{1-\frac{1}{2p}}
\end{equation}
and
\begin{equation}\label{GN-d2-four-thirds-p}
	\|u\|_{\frac{3}{4}}
	\le
	C_p M^{\frac{3}{4}}
	\|\nabla u^{\frac{p}{2}}\|_2^{\frac{1}{2p}}.
\end{equation}
Multiplying these inequalities gives
\begin{equation}\label{GN-d2-product-mass}
	\|u^{\frac{p}{2}}\|_4\|u\|_{\frac{4}{3}}
	\le
	C_pM\|\nabla u^{p/2}\|_2.
\end{equation}
Therefore,
\begin{equation}\label{Lp-first-term-d2}
	\begin{split}
		\left|
		\int_{\mathbb R^2}
		u^{\frac{p}{2}}\nabla u^{\frac{p}{2}}\cdot
		((\nabla K-\Psi)*u) \, \mathrm{d}x
		\right|
		&\le
		\|u^{\frac{p}{2}}\|_4
		\|\nabla u^{\frac{p}{2}}\|_2
		\|(\nabla K-\Psi)*u\|_4\\
		&\le
		\delta
		\|u^{\frac{p}{2}}\|_4
		\|u\|_{\frac{4}{3}}
		\|\nabla u^{\frac{p}{2}}\|_2\\
		&\le
		C_p\delta M
		\|\nabla u^{\frac{p}{2}}\|_2^2.
	\end{split}
\end{equation}
For the term containing \(\Psi\), we use
\begin{equation}\label{GN-d2-L2-p}
	\|u^{\frac{p}{2}}\|_4
	\le
	C_p\|u\|_2^{\frac{1}{2}}
	\|\nabla u^{\frac{p}{2}}\|_2^{1-\frac1p}
\end{equation}
and
\begin{equation}\label{GN-d2-L4-p}
	\|u\|_4
	\le
	C_p\|u\|_2^{1/2}
	\|\nabla u^{p/2}\|_2^{\frac1p}.
\end{equation}
Hence,
\begin{equation}\label{GN-d2-product-L2}
	\|u^{\frac{p}{2}}\|_4\|u\|_4
	\le
	C_p\|u\|_2
	\|\nabla u^{\frac{p}{2}}\|_2.
\end{equation}
Young's convolution inequality then gives
\begin{equation}\label{Lp-second-term-d2}
	\begin{split}
		\left|
		\int_{\mathbb R^2}
		u^{\frac{p}{2}}\nabla u^{\frac{p}{2}}\cdot(\Psi*u) \, \mathrm{d}x
		\right|
		&\le
		\|u^{\frac{p}{2}}\|_4
		\|\nabla u^{\frac{p}{2}}\|_2
		\|\Psi*u\|_4\\
		&\le
		\|\Psi\|_1
		\|u^{\frac{p}{2}}\|_4
		\|u\|_4
		\|\nabla u^{\frac{p}{2}}\|_2\\
		&\le
		C_p\|\Psi\|_1\|u(t)\|_2
		\|\nabla u^{\frac{p}{2}}\|_2^2.
	\end{split}
\end{equation}
Combining \eqref{Lp-first-term-d2} and
\eqref{Lp-second-term-d2} with
\eqref{Lp-energy-identity}, we obtain
\[
\begin{split}
	\frac{1}{p}\frac{d}{dt}Y_p(t)
	\leq
	\Bigg[
	-\frac{4(p-1)}{p^2}
	+\frac{2(p-1)}p C_p
	\Big(
	\delta M+\|\Psi\|_1\|u(t)\|_2
	\Big)
	\Bigg]
	\|\nabla u^{\frac{p}{2}}(t)\|_2^2.
\end{split}
\]
Since
\[
\lim_{t\to\infty}\|u(t)\|_2=0,
\]
we may again choose \(\delta>0\) sufficiently small and then \(T_p>0\) sufficiently large so that \eqref{Lp-dissipation} holds also in dimension \(d=2\).

We next use the generalized Nash inequality
\begin{equation}\label{generalized-Nash-p}
	\|\nabla u^{\frac{p}{2}}\|_2^2
	\ge
	C_{d,p}
	M^{-\frac{2p}{d(p-1)}}
	\left(\|u\|_p^p\right)^{
		1+\frac{2}{d(p-1)}
	}.
\end{equation}
Indeed, this is the Gagliardo--Nirenberg--Nash inequality applied to \(u^{p/2}\), with the \(L^{\frac{2}{p}}\)-quantity determined by the mass:
\[
\int_{\mathbb R^d}
\left(u^{\frac{p}{2}}\right)^{\frac{2}{p}} \, \mathrm{d}x
=
\int_{\mathbb R^d}u \, \mathrm{d}x=M.
\]
Substitution of \eqref{generalized-Nash-p} into
\eqref{Lp-dissipation} yields
\begin{equation}\label{Lp-final-ODE}
	\frac{d}{dt}Y_p(t)
	\le
	-C_{d,p,M}
	Y_p(t)^{
		1+\frac{2}{d(p-1)}
	},
	\qquad t\ge T_p.
\end{equation}
Let
\[
\alpha_p=\frac{2}{d(p-1)}.
\]
Integrating differential inequality \eqref{Lp-final-ODE}, we obtain
\[
Y_p(t)^{-\alpha_p}
\ge
Y_p(T_p)^{-\alpha_p}
+
\alpha_p C_{d,p,M}(t-T_p),
\qquad t\ge T_p.
\]
Consequently,
\[
Y_p(t)
\le
C_p(t-T_p)^{-\frac{d(p-1)}2},
\qquad t\ge T_p+1.
\]
Since \(Y_p(t)=\|u(t)\|_p^p\), taking the \(p\)-th root gives
\begin{equation}\label{optimal-Lp-decay-pgt2}
	\|u(t)\|_p
	\le
	C_p t^{-\frac{d(p-1)}{2p}}
	=
	C_p t^{-\frac d2\left(1-\frac1p\right)},
	\qquad t\ge T_p+1.
\end{equation}
The {\it a priori} bounds from Lemma~\ref{lem:bound} allow us to enlarge \(C_p\) so that estimate \eqref{optimal-Lp-decay-pgt2} holds for every \(t>0\).

For \(1<p\le2\), the estimate follows from interpolation between \(L^1\) and \(L^2\).
For \(p=1\), estimate \eqref{optimal-Lp-decay-pgt2} reduces to the mass conservation identity
\(
\|u(t)\|_1=M.
\)
For $p=\infty$, the optimal decay of $\|u(t)\|_\infty$ can be directly obtained via the Duhamel representation of solution 
to problem \eqref{eq:main}, in the same way as in the proof of Lemma \ref{lem:nonlinear-term:heat}, below.
\end{proof}

\subsection{Decay for singular-repulsive kernels}

The main result of this subsection, formulated in the following theorem, states that the heat-like $L^p$-decay estimates can also be achieved for certain singular kernels. These singular kernels can be decomposed into a sum of a repulsive part 
(see Definition \ref{def:repulsive kernel})
and a part that is small in the critical Marcinkiewicz space $L^{d,\infty}(\mathbb{R}^d)$.

\begin{theorem}\label{thm:decay:repulsive}
Let \(d\ge2\), and let \(u\) be a nonnegative solution of
\eqref{eq:main}--\eqref{ini:L1Linf}.
Assume also \eqref{nablaK:ass}.
Suppose that \(K=K_1+K_2\), where
\[
\Delta K_1\le0
\quad\text{in the sense of distributions},
\qquad
\nabla K_2\in L^{d,\infty}(\mathbb R^d).
\]
Then there exists \(D_d>0\) such that, if
\[
\|\nabla K_2\|_{L^{d,\infty}}\|u_0\|_1\le D_d,
\]
then, for every \(q\in[1,\infty)\),
\[
\|u(t)\|_q
\le C_q\,t^{-\frac d2(1-\frac1q)},
\qquad t>0,
\]
where \(C_q\) is independent of \(t\).
\end{theorem}

\begin{proof}
Set
\[
M=\|u_0\|_1,
\qquad
A=\|\nabla K_2\|_{L^{d,\infty}},
\]
and introduce
\[
a=\frac{2d}{d-1},
\qquad
b=\frac{2d}{2d-1}.
\]
By the weak Young inequality,
\begin{equation}\label{eq:wy-compact}
\|\nabla K_2*u\|_{2d}
\le C_d A\|u\|_b.
\end{equation}
In the following, we also 
use the Gagliardo--Nirenberg inequalities
\begin{equation}\label{eq:GN-compact}
\|u\|_a
 \le C_d\|\nabla u\|_2^{\frac{d+1}{d+2}}
          \|u\|_1^{\frac1{d+2}},
\qquad
\|u\|_b
 \le C_d\|\nabla u\|_2^{\frac1{d+2}}
          \|u\|_1^{\frac{d+1}{d+2}},
\end{equation}
which are valid for \(d\ge2\).

\medskip
\noindent
\textit{Step 1: \(L^2\)-decay.}
Multiplying equation \eqref{eq:main} by \(u\) gives
\[
\frac12\frac d{dt}\|u\|_2^2
=
-\|\nabla u\|_2^2
-\int_{\mathbb R^d}
u\nabla u\cdot(\nabla K_1*u) \, \mathrm{d}x
-\int_{\mathbb R^d}
u\nabla u\cdot(\nabla K_2*u) \, \mathrm{d}x.
\]
Since \(u\ge0\) and \(\Delta K_1\le0\),
\[
-\int_{\R^d} u\nabla u\cdot(\nabla K_1*u) \, \mathrm{d}x
=
\frac12\int_{\R^d} u^2(\Delta K_1*u) \, \mathrm{d}x\le0.
\]
Moreover, by the H\"older inequality, inequalities \eqref{eq:wy-compact} and 
\eqref{eq:GN-compact}, we arrive at
\begin{align*}
\left|
\int_{\R^d} u\nabla u\cdot(\nabla K_2*u) \, \mathrm{d}x
\right|
&\le
\|\nabla u\|_2\|u\|_a\|\nabla K_2*u\|_{2d} \\
&\le
C_d A\|\nabla u\|_2\|u\|_a\|u\|_b \\
&\le
C_d AM\|\nabla u\|_2^2.
\end{align*}
Choose \(D_d>0\) so that \(C_dD_d\le1/2\). Then
\begin{equation}\label{eq:L2-energy-compact}
\frac d{dt}\|u\|_2^2+c_d\|\nabla u\|_2^2\le0.
\end{equation}
The Nash inequality
\[
\|u\|_2^{2+\frac4d}
\le C_d M^{\frac4d}\|\nabla u\|_2^2
\]
therefore gives
\[
\frac d{dt}\|u\|_2^2
+c_dM^{-\frac4d}
\bigl(\|u\|_2^2\bigr)^{1+\frac2d}
\le0.
\]
Consequently,
\begin{equation}\label{eq:L2-decay-compact}
\|u(t)\|_2\le C_dM\,t^{-\frac{d}{4}}.
\end{equation}
Interpolation with the conserved \(L^1\)-norm yields
\begin{equation}\label{eq:low-p-decay}
\|u(t)\|_q
\le C_qt^{-\frac d2(1-\frac1q)},
\qquad 1\le q\le2.
\end{equation}
In particular,
\begin{equation}\label{eq:b-decay}
\|u(t)\|_b\le C M t^{-\frac{1}{4}}.
\end{equation}

\medskip
\noindent
\textit{Step 2: dyadic \(L^p\)-estimates.}
Let \(p>2\), set
\[
v=u^{\frac{p}{2}},
\qquad
y(t)=\|u(t)\|_p^p=\|v(t)\|_2^2.
\]
Multiplying the equation by \(pu^{p-1}\), and using
\(\Delta K_1*u\le0\), gives
\begin{equation}\label{eq:p-energy-start}
y'
+\frac{4(p-1)}p\|\nabla v\|_2^2
\le
2(p-1)
\left|
\int_{\R^d} v\nabla v\cdot(\nabla K_2*u) \, \mathrm{d}x
\right|.
\end{equation}
By H\"older, \eqref{eq:wy-compact}, and
\[
\|v\|_a
\le C_d\|\nabla v\|_2^{\frac{1}{2}}\|v\|_2^{\frac{1}{2}},
\]
we obtain
\begin{align*}
\left|
\int_{\R^d} v\nabla v\cdot(\nabla K_2*u) \, \mathrm{d}x
\right|
&\le
C_d A\|\nabla v\|_2\|v\|_a\|u\|_b \\
&\le
C_dA\|u\|_b
\|\nabla v\|_2^{\frac{3}{2}}y^{\frac{1}{4}}.
\end{align*}
The Young inequality and \eqref{eq:b-decay} imply
\begin{equation}\label{eq:p-energy-compact}
y'+c_p\|\nabla v\|_2^2
\le C_pA^4\|u\|_b^4y
\le \frac{C_p(AM)^4}{t}\,y.
\end{equation}

Assume inductively that
\[
\|u(t)\|_{\frac{p}{2}}
\le C_{p/2}
t^{-\frac d2(1-\frac2p)}.
\]
The Nash inequality applied to \(v\) gives
\[
\|\nabla v\|_2^2
\ge
C_d
\frac{y^{1+\frac2d}}
{\|v\|_1^{4/d}}
=
C_d
\frac{y^{1+\frac2d}}
{\|u\|_{p/2}^{2p/d}}
\ge
c_p t^{p-2}y^{1+\frac2d}.
\]
Hence
\begin{equation}\label{eq:ode-compact}
y'
+c_pt^{p-2}y^{1+\frac2d}
\le
\frac{B_p}{t}y.
\end{equation}
Let
\[
\alpha=\frac d2(p-1),
\qquad
z(t)=t^\alpha y(t).
\]
Using
\[
p-2-\frac{2\alpha}{d}=-1,
\]
we infer from \eqref{eq:ode-compact} that
\[
z'(t)
\le
\frac1t
\left((\alpha+B_p)z-c_pz^{1+\frac2d}\right).
\]
Since \(u_0\in L^1 (\R^d)\cap L^\infty(\R^d)\), \(y(t)\) is bounded near
\(t=0\), and therefore \(z(t)\to0\) as \(t\downarrow0\).
The preceding differential inequality then implies
\[
z(t)\le
\left(\frac{\alpha+B_p}{c_p}\right)^{\frac{d}{2}}.
\]
Thus
\[
\|u(t)\|_p^p=y(t)
\le C_pt^{-\frac d2(p-1)},
\]
or equivalently,
\[
\|u(t)\|_p
\le
C_pt^{-\frac d2(1-\frac1p)}.
\]

Starting with \(p=2\), this argument gives the estimate for all
dyadic exponents \(p=2^k\). Interpolation between \(L^1\) and a
sufficiently large dyadic \(L^p\)-space yields the claimed estimate
for every finite \(q\ge1\).
\end{proof}

\section{Heat kernel asymptotics}

 In this section, we complete the proofs of Theorems \ref{thm:regular:asymptotics:intro} and \ref{thm:heat:kernel:intro} using the $L^p$-decay estimates derived previously. We analyze solutions to the nonlinear Cauchy problem \eqref{eq:main} defined by the Duhamel representation \eqref{duhamel}. 
Our primary objective is to establish the conditions under which the nonlinear integral term over the interval $[0, t]$ decays faster than the heat kernel.

\begin{lemma}\label{lem:nonlinear-term:heat}
Let \(d\geq 2\), and suppose that one of the following assumptions holds:
\begin{enumerate}
    \item
    \(\nabla K\in L^{r,\infty}(\mathbb{R}^d)\) for some
    \(r\in(1,d)\);
    \item
    \(\nabla K\in L^d(\mathbb{R}^d)\).
\end{enumerate}
Assume that, for every \(q\in[1,\infty)\), there exists a constant
\(C_q>0\) such that
\begin{equation}\label{lem:heat-decay}
    \|u(t)\|_q
    \leq C_q t^{-\frac d2\left(1-\frac1q\right)}
    \qquad \text{for all } t>0.
\end{equation}
Then, for every \(p\in[1,\infty]\),
\begin{equation}\label{eq:late-nonlinear-term}
    t^{\frac d2\left(1-\frac1p\right)}
    \left\|
        \int_{0}^t
        \nabla\cdot e^{(t-s)\Delta}
        \bigl(u(s)(\nabla K*u(s))\bigr) \, \mathrm{d}s
    \right\|_p
    \longrightarrow 0
\end{equation}
as \(t\to\infty\).
\end{lemma}

\begin{proof}
Fix $p \in [1,\infty]$ and decompose 
$$N_p(t) := \left\| \int_0^t \nabla \cdot e^{(t-s)\Delta} \bigl( u(s)(\nabla K * u(s)) \bigr)\, \mathrm{d}s \right\|_p \le N_p^{(1)}(t) + N_p^{(2)}(t)$$ over the subintervals $[0, t/2]$ and $[t/2, t]$, respectively.

\smallskip
\noindent\textit{Step 1: Estimate on $[0,t/2]$.}
Using the standard heat semigroup bound $\|\nabla e^{\tau\Delta} f\|_p \le C \tau^{-\frac{1}{2}-\frac{d}{2}(1-\frac{1}{p})} \|f\|_1$ and $t-s \ge t/2$, we obtain
\begin{equation}\label{eq:early-scaled}
t^{\frac{d}{2}(1-\frac{1}{p})} N_p^{(1)}(t) \le C t^{-\frac{1}{2}} \int_0^{t/2} \|u(s)(\nabla K * u(s))\|_1 \, \mathrm{d}s.
\end{equation}

\textit{Case 1: $\nabla K \in L^{r,\infty}(\mathbb{R}^d)$ with $1 < r < d$.}
Choosing $b > r$ and setting $a^{-1} + b^{-1} = 1$, $1 + b^{-1} = r^{-1} + m^{-1}$, 
the H\"older and the weak Young inequalities along with decay bound \eqref{lem:heat-decay} yield
\[
\|u(s)(\nabla K * u(s))\|_1 \le C \|\nabla K\|_{L^{r,\infty}} \|u(s)\|_a \|u(s)\|_m \le C s^{-\frac{d}{2r}} \qquad \text{for } s \ge 1.
\]
Since $C_0 := \int_0^1 \|u(s)(\nabla K * u(s))\|_1 \, \mathrm{d}s < \infty$ and $\beta := \frac{d}{2r} > \frac{1}{2}$, we have $t^{-1/2} \int_1^{t/2} s^{-\beta}\, \mathrm{d}s \to 0$ as $t \to \infty$. Consequently,
\begin{equation}\label{eq:early-weak-limit}
t^{\frac{d}{2}(1-\frac{1}{p})} N_p^{(1)}(t) \to 0.
\end{equation}

\textit{Case 2: $\nabla K \in L^d(\mathbb{R}^d)$.}
Given $\varepsilon > 0$, decompose $\nabla K = k_1 + k_2$ with $k_1 \in C_c^\infty(\mathbb{R}^d) \subset L^1(\mathbb{R}^d)$ and $\|k_2\|_d \le \varepsilon$. For $s \ge 1$, 
the Young and the H\"older inequalities together with \eqref{lem:heat-decay} yield
\[
\|u(s)(k_1 * u(s))\|_1 \le \|k_1\|_1 \|u(s)\|_2^2 \le C_{k_1} s^{-\frac{d}{2}},
\]
and setting $d' = d/(d-1)$,
\[
\|u(s)(k_2 * u(s))\|_1 \le \varepsilon \|u(s)\|_{d'} \|u(s)\|_1 \le C \varepsilon s^{-\frac{1}{2}}.
\]
Inserting these bounds into \eqref{eq:early-scaled} gives
\begin{equation}\label{eq:early-strong-limsup}
\limsup_{t\to\infty} t^{\frac{d}{2}(1-\frac{1}{p})} N_p^{(1)}(t) \le C \varepsilon \limsup_{t\to\infty} t^{-\frac{1}{2}} \int_1^{t/2} s^{-\frac{1}{2}}\, \mathrm{d}s \le C \varepsilon.
\end{equation}

\smallskip
\noindent\textit{Step 2: Estimate on $[t/2,t]$.}
Set $q = p$ if $p < \infty$ and choose any finite $q > d$ if $p = \infty$, so that $\theta := \frac{1}{2} + \frac{d}{2}(\frac{1}{q} - \frac{1}{p}) < 1$. The heat semigroup estimate yields
\begin{equation}\label{eq:late-general}
N_p^{(2)}(t) \le C \int_{t/2}^t (t-s)^{-\theta} \|u(s)(\nabla K * u(s))\|_q\, \mathrm{d}s.
\end{equation}

\textit{Case 1: $\nabla K \in L^{r,\infty}(\mathbb{R}^d)$ with $1 < r < d$.}
Selecting $a,b,m > 1$ with $q^{-1} = a^{-1} + b^{-1}$ and $1 + b^{-1} = r^{-1} + m^{-1}$, \eqref{lem:heat-decay} implies
\[
\|u(s)(\nabla K * u(s))\|_q \le C \|u(s)\|_a \|u(s)\|_m \le C s^{-\frac{d}{2}(1-\frac{1}{q}+\frac{1}{r})}.
\]
Substituting this into \eqref{eq:late-general} and integrating over $s \in [t/2,t]$ gives $N_p^{(2)}(t) \le C t^{-\frac{d}{2}(1-\frac{1}{p})} t^{\frac{1}{2}-\frac{d}{2r}}$. Since $\frac{1}{2} - \frac{d}{2r} < 0$, we conclude
\begin{equation}\label{eq:late-weak-limit}
t^{\frac{d}{2}(1-\frac{1}{p})} N_p^{(2)}(t) \to 0.
\end{equation}
Combining \eqref{eq:early-weak-limit} and \eqref{eq:late-weak-limit} settles the result under Case 1.

\textit{Case 2: $\nabla K \in L^d(\mathbb{R}^d)$.}
Decompose $\nabla K = k_1 + k_2$ as in Step 1. Choosing exponents $a,b,m > 1$ satisfying $q^{-1} = a^{-1} + b^{-1}$ and $1 + b^{-1} = d^{-1} + m^{-1}$, we bound
\[
\|u(s)(k_2 * u(s))\|_q \le \varepsilon \|u(s)\|_a \|u(s)\|_m \le C \varepsilon s^{-\frac{d}{2}(1-\frac{1}{q})-\frac{1}{2}},
\]
which implies $\limsup_{t\to\infty} t^{\frac{d}{2}(1-\frac{1}{p})} N_{p,2}^{(2)}(t) \le C \varepsilon$. For $k_1 \in L^1(\mathbb{R}^d)$,
\[
\|u(s)(k_1 * u(s))\|_q \le \|k_1\|_1 \|u(s)\|_a \|u(s)\|_b \le C_{k_1} s^{-\frac{d}{2}(2-\frac{1}{q})},
\]
yielding $t^{\frac{d}{2}(1-\frac{1}{p})} N_{p,1}^{(2)}(t) \le C_{k_1} t^{-\frac{d-1}{2}} \to 0$. Therefore,
\[
\limsup_{t\to\infty} t^{\frac{d}{2}(1-\frac{1}{p})} N_p(t) \le C \varepsilon.
\]
Since $\varepsilon > 0$ is arbitrary, $t^{\frac{d}{2}(1-\frac{1}{p})} N_p(t) \to 0$, establishing \eqref{eq:late-nonlinear-term} for all $p \in [1,\infty]$.
\end{proof}

\begin{rem}
    In dimension $d=1$, the convergence in \eqref{eq:late-nonlinear-term} can be established for every kernel satisfying $\nabla K \in L^1(\mathbb{R})$ together with the additional condition $\int_\mathbb{R} \nabla K \, \d x = 0$. This result forms the basis for the findings in \cite{KS10-spikes} concerning the large-time asymptotics (given by the heat kernel \eqref{heat kernel}) of solutions to problem \eqref{eq:main}--\eqref{ini:L1Linf}, considered in one space dimension.
\end{rem}

\begin{proof}[Proof of Theorems \ref{thm:regular:asymptotics:intro} and \ref{thm:heat:kernel:intro}.]
We systematically apply the Duhamel formula \eqref{duhamel}. Given relation \eqref{heat:kernel:asymptotics} and the conservation of mass ($M = \text{const}$), the linear term $e^{{t}\Delta}u_0$ behaves asymptotically like $M G(t)$. Since Theorem \ref{thm:regular:decay} provides decay estimates \eqref{lem:heat-decay}, Lemma~\ref{lem:nonlinear-term:heat} ensures that the nonlinear term decays faster. This completes the proof of Theorem \ref{thm:regular:asymptotics:intro}. The proof of Theorem \ref{thm:heat:kernel:intro} is analogous and relies on the decay estimates established in Theorem \ref{thm:decay:repulsive}.
\end{proof}

\section{Self-similar solutions and self-similar asymptotics}

Next, we consider the Cauchy problem \eqref{eq:main}--\eqref{ini:L1Linf} in $\R^d$ for $d \ge 2$ and with the critical logarithmic interaction kernel
\begin{equation}\label{K=logx}
    K = -\log |x| \quad \text{satisfying} \quad \nabla K(x) = -\frac{x}{|x|^2} \in L^{d,\infty}(\mathbb{R}^d) \setminus L^d(\mathbb{R}^d).
\end{equation}

\begin{rem}\label{rem:scaling}
The first equation in \eqref{eq:main} is invariant under transformation 
\begin{align}\label{def:rescaled-solution}
u_\lambda(x,t):=\lambda^d u(\lambda x,\lambda^2 t),
\qquad x\in\R^d,\ \ t>0
\end{align}
for every $\lambda > 0$. Indeed,
note first that
\begin{align*}
\partial_t u_\lambda(x,t)
&= \lambda^{d+2} (\partial_t u)(\lambda x,\lambda^2 t), \quad
\Delta u_\lambda(x,t)
= \lambda^{d+2} (\Delta u)(\lambda x,\lambda^2 t).
\end{align*}
Moreover, since
\(
\nabla K(x) = -{x}/{|x|^2}
\)
is homogeneous of degree $-1$, we have
\begin{align*}
(\nabla K * u_\lambda)(x,t)
&= \lambda (\nabla K * u)(\lambda x,\lambda^2 t),
\end{align*}
which yields
\begin{align*}
\nabla \cdot \bigl( u_\lambda (\nabla K * u_\lambda) \bigr)(x,t)
&= \lambda^{d+2} \Bigl( \nabla \cdot \bigl( u (\nabla K * u) \bigr) \Bigr) (\lambda x,\lambda^2 t).
\end{align*}
Consequently, $u_\lambda$ satisfies the original equation \eqref{eq:main}.
\end{rem}





\subsection{Preliminary estimates}
We begin our analysis of the asymptotic properties of solutions by establishing their optimal estimates.

\begin{prop} \label{prop:critical:optimal:decay}
Let $u = u(x,t)$ be a non-negative solution to the Cauchy problem \eqref{eq:main}-\eqref{ini:L1Linf}
with the repulsive logarithmic kernel \eqref{K=logx}.
 Then, for every $p \in [1, \infty]$, there exist constants $C_1 > 0$ and $C_2>0$, independent of $t$, such that
\begin{align}\label{eq:Lp-decay}
\|u(t)\|_{p} \le C_1 t^{-\frac{d}{2}\left(1-\frac{1}{p}\right)}
\end{align}
and
\begin{align}\label{eq:gradient:Lp-decay}
\|\nabla u(t)\|_{p} \le C_2 t^{-\frac{d}{2}\left(1-\frac{1}{p}\right)-\frac{1}{2}} 
\end{align}
for all $t>0$. Moreover, if $\int_{\R^d} |x|^2 u_0 (x)  \, \mathrm{d}x < \infty$, then for $M=\int_{\R^d} u_0(x) \, \mathrm{d}x$,
\begin{equation}\label{moment:evolution}
   \int_{\R^d} |x|^2 u (x,t)  \, \mathrm{d}x=\int_{\R^d} |x|^2 u_0 (x)  \, \mathrm{d}x  + (2dM + M^2)t
\end{equation}
for all $t\geq 0$.
\end{prop}

We first recall two convolution estimates used in the proof.

\begin{lemma}\label{kernel-estimate}
Let $d \ge 2$ and \(
\nabla K(x) = -{x}/{|x|^2}.
\)
\begin{enumerate}
    \item \cite[Theorem 4.5.3]{H15}
For all $1 < q < p < \infty$ such that  $\frac{1}{p}=\frac{1}{q}-\frac{d-1}{d}$, 
there exists a constant $C=C(d,p,q)>0$ such that
\begin{equation} \label{ineq:kernel-estimate_1}
\|\nabla K * f\|_{p} \leq C\|f\|_{q}
\end{equation}
for all $f \in L^q(\R^d)$.

\item \cite[Lemma 4.5.4]{H15}
There exists a constant $C=C(d)>0$ such that
\begin{equation}\label{ineq:kernel-estimate_2}
\|\nabla K * f\|_{\infty} \leq C\|f\|_{1}^\frac{d-1}{d} \|f\|_{\infty}^\frac{1}{d} 
\end{equation}
for all $f \in L^1(\R^d) \cap L^\infty(\R^d)$.
\end{enumerate}
\end{lemma}

\begin{proof}[Proof of Proposition \ref{prop:critical:optimal:decay}]
{\it Step 1.~$L^p$-decay of $u$.}~We first note that the kernel $K(x) = -\log |x|$ is repulsive in the sense of Theorem~\ref{thm:decay:repulsive}. Indeed, for $d \ge 3$, a direct computation gives 
$\Delta K(x) = -(d-2)/{|x|^2} \le 0$, whereas for $d = 2$, the logarithmic kernel satisfies $\Delta(-\log |x|) = -2\pi \delta_0 \le 0$ in the sense of distributions. Thus, the decay estimates \eqref{eq:Lp-decay} for $p \in [1, \infty)$ follow immediately from Theorem~\ref{thm:decay:repulsive} applied with $K_1(x) = -\log |x|$ and $K_2 \equiv 0$.

We obtain this estimate in the case $p=\infty$ by using the following direct consequence of the Duhamel principle \eqref{duhamel}: 
\begin{equation}\label{duhamel:t/2}
 u(t) = e^{\frac{t}{2}\Delta}u(t/2) + \int_{t/2}^t \nabla \cdot e^{(t-s)\Delta} (u(s) \nabla K * u(s))  \, \mathrm{d}s.
 \end{equation}
Fix a finite $p>d$ and choose $m$ by
\[
 \frac1m=\frac{d-1}{d}+\frac1{2p}.
\]
Then $1<m<d/(d-1)$ and Lemma~\ref{kernel-estimate} (1), followed by
the H\"older inequality, give
\[
 \|u(\nabla K*u)\|_p
 \le \|u\|_{2p}\|\nabla K*u\|_{2p}
 \le C\|u\|_{2p}\|u\|_m.
\]
The finite-exponent decay already proved implies
\[
 \|u(s)(\nabla K*u(s))\|_p
 \le C s^{-\frac d2(1-\frac1p)-\frac12}.
\]
Consequently, \eqref{duhamel:t/2} yields
\begin{align*}
 \|u(t)\|_\infty
 &\le CMt^{-d/2}
   +C\int_{t/2}^t(t-s)^{-\frac12-\frac d{2p}}
       s^{-\frac d2(1-\frac1p)-\frac12}\,\d s\\
 &\le C t^{-d/2}.
\end{align*}
The time integral is finite because $p>d$.

\medskip
\noindent
\textit{Step 2: Decay of the gradient.}~First, note that by the definition of $u_\lambda$ in \eqref{def:rescaled-solution} and the change of variables $y = \lambda x$, we have for every $p \in [1,\infty]$
\[
\|u_\lambda(t)\|_p = \lambda^{d\left(1-\frac{1}{p}\right)} \|u(\lambda^2 t)\|_p.
\]
Using the decay estimate \eqref{eq:Lp-decay} at time $\lambda^2 t$, we obtain
\begin{equation}\label{u:decay:rescal}
\|u_\lambda(t)\|_p \le C \lambda^{d\left(1-\frac{1}{p}\right)} (\lambda^2 t)^{-\frac{d}{2}\left(1-\frac{1}{p}\right)} = C t^{-\frac{d}{2}\left(1-\frac{1}{p}\right)},
\end{equation}
where the constant $C$ is independent of $t$ and $\lambda$.

We now use the scaling \eqref{def:rescaled-solution} together with estimate \eqref{u:decay:rescal} to prove the gradient decay estimate \eqref{eq:gradient:Lp-decay}. Set
\[
V_\lambda = \nabla K * u_\lambda.
\]
For $0 < t \le 1$, the mild formulation of problem \eqref{eq:main} for the rescaled function $u_\lambda$ starting at time $1$ can be written as
\begin{align*}
u_\lambda(t+1) &= e^{t\Delta} u_\lambda(1) + \int_0^t e^{(t-s)\Delta} \bigl( \nabla u_\lambda \cdot V_\lambda + u_\lambda \nabla \cdot V_\lambda \bigr)(s+1)\, \mathrm{d}s.
\end{align*}
Therefore, for every $p \in [1,\infty]$,
\begin{align}
\|\nabla u_\lambda(t+1)\|_p &\le C t^{-1/2} \|u_\lambda(1)\|_p \notag \\
&\quad + C \int_0^t (t-s)^{-1/2} \|V_\lambda(s+1)\|_\infty \|\nabla u_\lambda(s+1)\|_p \, \mathrm{d}s \notag \\
&\quad + C \int_0^t (t-s)^{-1/2} \|u_\lambda(s+1) \nabla \cdot V_\lambda(s+1)\|_p \, \mathrm{d}s. \label{eq:gradient-local-corrected}
\end{align}
By inequality \eqref{ineq:kernel-estimate_2},
\[
\|V_\lambda(\tau)\|_\infty \le C_d \|u_\lambda(\tau)\|_1^{1-\frac{1}{d}} \|u_\lambda(\tau)\|_\infty^{\frac{1}{d}}.
\]
Thus, estimate \eqref{u:decay:rescal} yields
\begin{equation}\label{eq:uniform-V}
\sup_{\lambda>0} \sup_{1 \le \tau \le 2} \|V_\lambda(\tau)\|_\infty \le C.
\end{equation}

We next estimate $\nabla \cdot V_\lambda$. If $d = 2$, then $\nabla \cdot V_\lambda = -2\pi u_\lambda$, and hence
\[
\|u_\lambda \nabla \cdot V_\lambda\|_p \le 2\pi \|u_\lambda\|_\infty \|u_\lambda\|_p.
\]
If $d \ge 3$, then $\nabla \cdot V_\lambda = -(d-2) |x|^{-2} * u_\lambda$. An argument analogous to the one leading to inequality \eqref{ineq:kernel-estimate_2} gives
\[
\bigl\| |x|^{-2} * f \bigr\|_\infty \le C_d \|f\|_1^{1-\frac{2}{d}} \|f\|_{\infty}^{\frac{2}{d}}.
\]
Therefore,
\[
\|u_\lambda \nabla \cdot V_\lambda\|_p \le C_d \|u_\lambda\|_p \|u_\lambda\|_1^{1-\frac{2}{d}} \|u_\lambda\|_\infty^{\frac{2}{d}}.
\]
In both cases, the $\lambda$-independent estimates \eqref{u:decay:rescal} yield
\begin{equation}\label{eq:uniform-div-V}
\sup_{\lambda>0} \sup_{1 \le \tau \le 2} \|u_\lambda(\tau) \nabla \cdot V_\lambda(\tau)\|_p \le C_p.
\end{equation}

Now, define $F_\lambda(t) = \|\nabla u_\lambda(t+1)\|_p$. Combining inequalities \eqref{eq:gradient-local-corrected}--\eqref{eq:uniform-div-V}, we obtain
\[
F_\lambda(t) \le C_p t^{-\frac{1}{2}} + C \int_0^t (t-s)^{-\frac{1}{2}} F_\lambda(s)\, \mathrm{d}s, \qquad 0 < t \le 1.
\]
The singular Gr\"onwall inequality then implies
\[
F_\lambda(t) \le C_p t^{-\frac{1}{2}}, \qquad 0 < t \le 1,
\]
where $C_p$ is independent of $\lambda$. In particular,
\begin{equation}\label{eq:gradient-rescaled-time-two}
\|\nabla u_\lambda(2)\|_p \le C_p
\end{equation}
for every $\lambda > 0$.

By the definition of $u_\lambda$ and a change of variables, we have
\[
\|\nabla u_\lambda(2)\|_p = \lambda^{d+1-\frac{d}{p}} \|\nabla u(2\lambda^2)\|_p.
\]
Given $T > 0$, setting $\lambda = (T/2)^{1/2}$ in estimate \eqref{eq:gradient-rescaled-time-two} gives
\[
\|\nabla u(T)\|_p \le C_p T^{-\frac{1}{2}\left(d+1-\frac{d}{p}\right)} = C_p T^{-\frac{d}{2}\left(1-\frac{1}{p}\right)-\frac{1}{2}}.
\]
This completes the proof of the gradient decay estimate \eqref{eq:gradient:Lp-decay} for every $p \in [1,\infty]$.

\medskip
\noindent
\textit{Step 3: Evolution of the second moment.}~We denote the second moment of the solution by 
\begin{align*}
\mu_2 (t) := \int_{\R^d}|x|^2 u(x,t)  \, \mathrm{d}x.
\end{align*}
Since $\nabla K = -{x}/{|x|^2}$ for $x \in \R^d$, we obtain the equation
\begin{align*}
\dfrac{d}{dt}\mu_2(t) &= 2dM - 2\int_{\R^d}u(x,t)x\cdot \nabla K * u(x,t)  \, \mathrm{d}x\\
&= 2dM + \int_{\R^d}\int_{\R^d} u(x,t)(x-y) \cdot \dfrac{x-y}{|x-y|^2}u(y,t)  \, \mathrm{d}y  \, \mathrm{d}x\\
&= 2dM + M^2,
\end{align*}
which is equivalent to equation \eqref{moment:evolution}.
\end{proof}

\subsection{Self-similar variables}

Let ${u}=u(x,t)$ be a solution of problem \eqref{eq:main}--\eqref{ini:L1Linf} with  mass $M> 0$ defined in \eqref{mass:intro}.
We introduce the self-similar variables
\begin{equation}\label{eq:ss-change}
\xi=\frac{x}{\sqrt t},\qquad s=\log t,
\end{equation}
and define the new function
\begin{equation}\label{eq:def-W}
W(\xi,s):=e^{\frac d2 s}{u}\left(e^{\frac{s}{2}}\xi,e^s\right)
=t^{\frac{d}{2}}{u}\left(t^{\frac{1}{2}}\xi ,t\right)
\end{equation}
or, equivalently,
\begin{equation}\label{eq:inverse-W}
{u}(x,t)=t^{-\frac{d}{2}}W\!\left(\frac{x}{\sqrt t},\log t\right).
\end{equation}
Note that, by the change of variables $x=e^{s/2}\xi$,
\begin{equation}\label{mass-for-W}
\int_{\mathbb R^d}W(\xi,s) \, \mathrm{d}\xi
=
\int_{\mathbb R^d}{u}(x,e^s) \, \mathrm{d}x
=
M
\quad \mathrm{for\ all}\ s\in \R.
\end{equation}
Moreover, since $\nabla K$ is homogeneous of degree $-1$, we have (see also Remark \ref{rem:scaling})
\begin{equation}\label{eq:kernel-scaling}
(\nabla K*{u})(e^{\frac{s}{2}}\xi,e^s)
=
e^{-\frac{s}{2}}(\nabla K*W(\cdot,s))(\xi).
\end{equation}
Therefore, a direct computation shows that \(W\) satisfies the equation
\begin{equation}\label{eq:FP}
\partial_s W
=
\Delta W+\frac12\nabla\cdot(\xi W)
+\nabla\cdot\bigl(W(\nabla K*W)\bigr)
\quad\text{for } \xi\in\mathbb R^d, \; s\in\mathbb R.
\end{equation}

\begin{lemma}\label{lem:estimates-W}
Let \(W\) be defined by \eqref{eq:def-W}, where \(u=u(x,t)\) is a nonnegative
solution of \eqref{eq:main}--\eqref{ini:L1Linf} with the logarithmic
kernel \eqref{K=logx}. Then, for every \(p\in[1,\infty]\),
\begin{equation}\label{eq1:estimate for W}
    \sup_{s\in\mathbb R}\|W(s)\|_{p}
    <\infty,
    \qquad
    \sup_{s\in\mathbb R}\|\nabla W(s)\|_{p}
    <\infty.
\end{equation}
Moreover, if
\(
    \int_{\mathbb R^d}|x|^2u_0(x)\,dx<\infty,
\)
then, for every  fixed \(s_0\in\mathbb R\),
\begin{equation}\label{eq2:estimate for W}
    \sup_{s\geq s_0}
    \int_{\mathbb R^d}|\xi|^2W(\xi,s)\, \mathrm{d}\xi
    \leq  e^{-s_0}
    \int_{\mathbb R^d}|x|^2u_0(x)\, \mathrm{d}x
    +2dM+M^2
\end{equation}
\end{lemma}

\begin{proof}
By definition \eqref{eq:def-W}, the change of variables
\(x=e^{s/2}\xi\) gives
\begin{align*}
    \|W(s)\|_p
    &=
    e^{\frac d2(1-\frac1p)s}
    \|u(e^s)\|_p.
\end{align*}
Using the decay estimate \eqref{eq:Lp-decay}, we obtain
\[
    \|W(s)\|_p
    \leq
    C
    e^{\frac d2(1-\frac1p)s}
    e^{-\frac d2(1-\frac1p)s}
    =C.
\]

Next,
\[
    \nabla_\xi W(\xi,s)
    =
    e^{\frac{d+1}{2}s}
    \nabla u\left(e^{\frac{s}{2}}\xi,e^s\right).
\]
Therefore, for \(1\leq p\le \infty\),
\[
    \|\nabla W(s)\|_p
    =
    e^{\left[
        \frac d2(1-\frac1p)+\frac12
    \right]s}
    \|\nabla u(e^s)\|_p.
\]
Thus, by gradient decay estimate \eqref{eq:gradient:Lp-decay}, we have
\(
    \|\nabla W(s)\|_p
    \leq C.
\)

It remains to estimate the second moment. Another change of variables
gives
\begin{align*}
    \int_{\mathbb R^d}|\xi|^2W(\xi,s)\,\mathrm{d}\xi
    &=
    e^{-s}
    \int_{\mathbb R^d}|x|^2u(x,e^s)\,\mathrm{d}x.
\end{align*}
Thus, by the second-moment identity \eqref{moment:evolution},
\begin{align*}
    \int_{\mathbb R^d}|\xi|^2W(\xi,s)\,\mathrm{d}\xi
    &\leq
    e^{-s_0}
    \int_{\mathbb R^d}|x|^2u_0(x)\,\mathrm{d}x
    +2dM+M^2
\end{align*}
 for every fixed \(s_0\in\mathbb R\) and all $s\geq s_0$.
\end{proof}

\begin{cor}\label{prop:compact for W}
Let \(W\) be defined by \eqref{eq:def-W}, and assume that
\(
\int_{\mathbb R^d}|x|^2u_0(x)\,dx<\infty.
\)
Then, for every \(s_0\in \R\) and every \(p\in[1,\infty)\), the set
\(
\{W(\cdot,s):s\geq s_0\}
\)
is relatively compact in \(L^p(\mathbb R^d)\). Moreover,
if $s_n\ge s_0$ and $W(s_n)\to W_0$ in $L^1(\R^d)$, then
\begin{align*}
    \lim_{n \to \infty} \int_{\R^d} |\xi|^2 |W(\xi,s_n) - W_0(\xi)| \, \mathrm{d}\xi = 0.
\end{align*}
\end{cor}

\begin{proof}
Fix $p \in [1,\infty)$, and let $\{s_n\} \subset [s_0,\infty)$ be an arbitrary sequence. Set
\[
W_n(\xi) := W(\xi,s_n).
\]
By \eqref{eq1:estimate for W},
\[
\sup_n \|W_n\|_{W^{1,p}(\mathbb{R}^d)} < \infty.
\]
Hence, by the Rellich--Kondrachov compactness theorem on bounded balls and a diagonal argument, there exists a subsequence (still denoted by $\{W_n\}$) and a function $W_\infty \in L^p_{\mathrm{loc}}(\mathbb{R}^d)$ such that
\[
W_n \to W_\infty \qquad \text{in } L^p_{\mathrm{loc}}(\mathbb{R}^d).
\]

It remains to control the tails. By the uniform second-moment bound \eqref{eq2:estimate for W},
\[
\sup_n \int_{|\xi| \ge R} W_n(\xi) \, \mathrm{d}\xi \le \frac{1}{R^2} \sup_n \int_{\mathbb{R}^d} |\xi|^2 W_n(\xi) \, \mathrm{d}\xi \le \frac{C}{R^2}.
\]
Thus,
\begin{equation}\label{eq:L1-tail-W}
\sup_n \int_{|\xi| \ge R} W_n(\xi) \, \mathrm{d}\xi \to 0 \qquad \text{as } R \to \infty.
\end{equation}

For $p = 1$, \eqref{eq:L1-tail-W} together with local $L^1$-convergence immediately yields convergence in $L^1(\mathbb{R}^d)$.

Now assume $1 < p < \infty$. Choose any $q > p$. By \eqref{eq1:estimate for W},
\[
\sup_n \|W_n\|_{q} \le C_q.
\]
Let $\alpha \in (0,1)$ be determined by
\[
\frac{1}{p} = \alpha + \frac{1-\alpha}{q}, \qquad \alpha = \frac{\frac{1}{p} - \frac{1}{q}}{1 - \frac{1}{q}}.
\]
Interpolating on the set $\{|\xi| \ge R\}$ gives
\begin{align*}
\|W_n\|_{L^p(|\xi| \ge R)} &\le \|W_n\|_{L^1(|\xi| \ge R)}^\alpha \|W_n\|_{L^q(|\xi| \ge R)}^{1-\alpha} \\
&\le C \|W_n\|_{L^1(|\xi| \ge R)}^\alpha.
\end{align*}
Therefore,
\[
\sup_n \|W_n\|_{L^p(|\xi| \ge R)} \to 0 \qquad \text{as } R \to \infty.
\]
Consequently, for every $R > 0$,
\begin{align*}
\|W_n - W_m\|_{p} &\le \|W_n - W_m\|_{L^p(B_R)} + \|W_n\|_{L^p(|\xi| \ge R)} + \|W_m\|_{L^p(|\xi| \ge R)}.
\end{align*}
By first choosing $R$ sufficiently large so that the last two terms are uniformly small, and then taking $n, m \to \infty$, we conclude that $\{W_n\}$ is Cauchy in $L^p(\mathbb{R}^d)$ because it converges in $L^p(B_R)$. Hence, $W_n \to W_\infty$ strongly in $L^p(\mathbb{R}^d)$. It remains to prove convergence in the second-moment weighted norm.

Let $(s_n)$ be a sequence such that $\|W_n - W_0\|_1 \to 0$ as $n \to \infty$. Since $W_n$ is a nonnegative function with mass $M$, we have
\begin{align*}
    W_0\ge0,
    \qquad
    \int_{\mathbb R^d}W_0(\xi)\, \mathrm{d}\xi=M.
\end{align*}
We first show the uniform tightness of the second moment for every $s_0 \in \R$:
\begin{align}\label{eq:uniform-second-moment-tail}
    \lim_{R\to\infty}
    \sup_{s\ge s_0}
    \int_{|\xi|>R}|\xi|^2W(\xi,s)\, \mathrm{d}\xi
    =
    0.
\end{align}
Let $\psi_R \in C^1(\R^d)$ be such that
\[
    \psi_R(\xi)=0 \quad\text{for } |\xi|\le R,
    \qquad
    \psi_R(\xi)=(|\xi|-R)^2 \quad\text{for } |\xi|\ge 2R,
\]
and such that
\[
    0\le \psi_R(\xi)\le C|\xi|^2,
    \qquad
    |\nabla\psi_R(\xi)|\le C\psi_R(\xi)^{\frac{1}{2}},
    \qquad
    |\Delta\psi_R(\xi)|\le C\mathbf 1_{\{|\xi|>R\}}.
\]
Testing the rescaled equation \eqref{eq:FP} by $\psi_R$, we obtain
\begin{align}
    \frac{d}{ds}\int_{\mathbb R^d}\psi_R W\, \mathrm{d}\xi
    &=
    \int_{\mathbb R^d}\Delta\psi_R W\, \mathrm{d}\xi
    -
    \frac12\int_{\mathbb R^d}\xi\cdot\nabla\psi_R W\, \mathrm{d}\xi\\
    &\quad
    -
    \int_{\mathbb R^d}
    \nabla\psi_R\cdot(\nabla K*W)W\, \mathrm{d}\xi .
    \label{eq:test-tail-moment}
\end{align}
Since Lemma~\ref{lem:estimates-W} allows us to apply the estimate \eqref{ineq:kernel-estimate_2} so that $\|\nabla K * W\|_\infty < \infty$ for every $s\in \R$, we get from the Young inequality
\begin{align}
    \left|
    \int_{\mathbb R^d}
    \nabla\psi_R\cdot(\nabla K*W)W\, \mathrm{d}\xi
    \right|
    &\le
    C\int_{\mathbb R^d}\psi_R^{\frac{1}{2}}W\, \mathrm{d}\xi
    \nonumber\\
    &\le
    \varepsilon\int_{\mathbb R^d}\psi_R W\, \mathrm{d}\xi
    +
    C(\varepsilon)\int_{|\xi|>R}W\, \mathrm{d}\xi .
    \label{eq:tail-drift-estimate}
\end{align}
Furthermore the test function $\psi_R$ satisfies 
\[
    -\frac12\xi\cdot\nabla\psi_R
    \le
    -c_0\psi_R
\]
for some $c_0>0$. Hence, using also the uniform second moment bound and choosing a suitable $\varepsilon > 0$, we arrive at
\begin{align}
    \frac{d}{ds}\int_{\mathbb R^d}\psi_R W\, \mathrm{d}\xi
    &\le
    -c\int_{\mathbb R^d}\psi_R W\, \mathrm{d}\xi
    +
    C\int_{|\xi|>R}W\, \mathrm{d}\xi
    \nonumber\\
    &\leq
    -c\int_{\mathbb R^d}\psi_R W\, \mathrm{d}\xi
    +
    \frac{C}{R^2}.
    \label{eq:tail-differential-ineq}
\end{align}
The Gr\"onwall inequality gives
\begin{align}
    \int_{\mathbb R^d}\psi_R W(\xi,s)\, \mathrm{d}\xi
    \le
    e^{-c(s-s_0)}
    \int_{\mathbb R^d}\psi_R W(\xi,s_0)\, \mathrm{d}\xi
    +
    \frac{C}{R^2}
\end{align}
for every $s\ge s_0$. Therefore 
using $\int_{\R^d} |\xi|^2W(\xi,s_0)\,\d\xi<\infty$ and dominated convergence theorem,
we conclude that
\begin{align}
    \sup_{s\ge s_0}
    \int_{\mathbb R^d}\psi_R W(\xi,s)\, \mathrm{d}\xi
    \le
    \int_{\mathbb R^d}\psi_R W(\xi,s_0)\, \mathrm{d}\xi
    +
    \frac{C}{R^2}
    \to 0
\end{align}
as $R\to\infty$. Since $\psi_R(\xi)\ge c|\xi|^2$ for $|\xi|\ge 2R$, this proves
\eqref{eq:uniform-second-moment-tail}.

Finally, we prove the convergence in the second moment. For every $R>0$, we have
\begin{align}
    \int_{\mathbb R^d}|\xi|^2|W_n-W_0|\, \mathrm{d}\xi
    &\leq
    R^2\|W_n-W_0\|_{1}
    +
    \int_{|\xi|>R}|\xi|^2W_n\, \mathrm{d}\xi
    +
    \int_{|\xi|>R}|\xi|^2W_0\, \mathrm{d}\xi .
    \label{eq:weighted-convergence-split}
\end{align}
The first term tends to zero as $n\to\infty$ due to the assumption and the second term is uniformly small for
large $R$ by \eqref{eq:uniform-second-moment-tail}. The third term is also small for
large $R$, by the Fatou lemma and \eqref{eq:uniform-second-moment-tail}. Hence gathering these information clearly enables us to finish the proof.
\end{proof}

\subsection{Lyapunov functional and LaSalle principle}
For the reader's convenience, we now recall the classical LaSalle invariance principle in the form used below.
A simple, classical proof of this principle can be obtained directly by following, for example, the arguments in 
\cite[Section 9]{CazenaveHaraux1998} and
\cite{Ball2016}.

\begin{rem}[LaSalle invariance principle]\label{LaSalle}
Let $X$ be a complete metric space, and let $S(\tau):X\to X $ for $\tau \ge 0$ be a continuous dynamical system. Let $\mathcal{F}:X\to\mathbb R$ be a continuous Lyapunov functional such that
for every $w\in X$, the map $\tau\mapsto \mathcal{F}(S(\tau)w)$
is nonincreasing on $[0,\infty)$ and bounded from below. For $w\in X$, define the $\omega$-limit set by
\begin{align}\label{eq:omega-limit-abstract}
\omega(w)
:=
\{
v\in X:\
S(\tau_n)w\to v \ \text{in }X
\ \text{for some sequence } \tau_n\to\infty\}.
\end{align}
Suppose that the trajectory $\bigcup_{\tau \ge 0} \{S(\tau)w\}$ is relatively compact in $X$. Then the set $\omega(w)$ is nonempty, compact, and invariant under the mapping $S(\tau)$ . Moreover, one obtains
\begin{align}\label{eq:lasalle-conclusion}
\omega(w)\subset \mathcal{E}
:=
\{
v\in X:\
\mathcal{F}(S(\tau)v)=\mathcal{F}(v)
\ \text{for every } \tau\ge0\}
\end{align} 
as well as
\begin{align}
{\rm dist}_{X}\{S(\tau)w, \omega(w)\}\to 0\quad \text{as}\ \tau\to\infty.
\end{align}
\end{rem}

We apply this principle to analyze the asymptotic behavior of the rescaled solution $W$ given by \eqref{eq:def-W} as $s \to \infty$. To this end, we first introduce a Lyapunov functional for equation \eqref{eq:FP}, which serves as a key tool for establishing the convergence of $W$ to a stationary solution.

\begin{lemma}\label{lem:Lyapunov for W}
Let $W$ be defined by \eqref{eq:def-W}, where $u = u(x,t)$ is a solution to problem \eqref{eq:main}--\eqref{ini:L1Linf} with the logarithmic kernel \eqref{K=logx}. Assume that
\begin{equation}\label{initial_moment}
\int_{\mathbb{R}^d} |x|^2 u_0(x)\, \mathrm{d}x < \infty.
\end{equation}
Then
\begin{equation}\label{energy-dissipation-rescal}
\frac{d}{ds} \mathcal{F}[W(s)] = -\mathcal{D}[W(s)],
\end{equation}
where
\begin{align*}
\mathcal{F}[W(s)] &:= \int_{\mathbb{R}^d} W(\xi,s) \log W(\xi,s) \, \mathrm{d}\xi + \int_{\mathbb{R}^d} \frac{|\xi|^2}{4} W(\xi,s) \, \mathrm{d}\xi + \frac{1}{2} \int_{\mathbb{R}^d} W(\xi,s) (K * W)(\xi,s) \, \mathrm{d}\xi, \\
\mathcal{D}[W(s)] &:= \int_{\mathbb{R}^d} W(\xi,s) \left| \nabla \left( \log W(\xi,s) + \frac{|\xi|^2}{4} + (K * W)(\xi,s) \right) \right|^2 \, \mathrm{d}\xi.
\end{align*}
Moreover, for every $s_0\in\R$, the functional $\mathcal{F}[W(s)]$ is bounded from below; more precisely, there exists a constant $C > 0$ depending only on the dimension $d$, the mass $M$, the initial moment \eqref{initial_moment}, and $s_0$ such that
\[
\mathcal{F}[W(s)] \ge -C \quad \text{for all } s \geq s_0.
\]
\end{lemma}

\begin{proof}
The functionals $\mathcal{F}$ and $\mathcal{D}$, as well as relation \eqref{energy-dissipation-rescal}, are analogues of the functionals in \eqref{entropy} and relation \eqref{entropy dyssipation}. 
Since their derivation is entirely analogous, we omit the explicit calculations. 
We emphasize that both $\mathcal{F}$ and $\mathcal{D}$ are well-defined and finite, owing to the properties of the solution $W(s)$ established in Lemma~\ref{lem:estimates-W}. 
Moreover, the map $s \mapsto \mathcal{F}[W(s)]$ is differentiable because $W(\xi,s)$ inherits the regularity of $u(x,t)$ stated in \eqref{regularity}.

Let us verify that the Lyapunov functional $\mathcal{F}$ is bounded from below. 
For the time-one Gaussian $G(\xi)  ={(4\pi)^{-d/2}} \exp(-{|\xi|^2}/{4})$ 
so that $\int_{\R^d} G(\xi) \, \mathrm{d}\xi = 1$, we apply the Jensen inequality to obtain
\begin{align*}
    \int_{\R^d}\dfrac{W}{M}\log \dfrac{W}{MG} \, \mathrm{d}\xi \ge 0.
\end{align*}
Therefore we get
\begin{align}\label{eq1: Lyapunov for W}
    \int_{\R^d} W\log W \, \mathrm{d}\xi + \dfrac{1}{4}\int_{\R^d}|\xi|^2 W \, \mathrm{d}\xi \ge M\log M -\dfrac{d}{2}(\log 4\pi)M =:C(M,d).
\end{align}
We next control the interaction term with $K(x) =-\log |x|$. Using the elementary inequality
\begin{align*}
    \log |\xi-\eta| \leq \dfrac{1}{2}\log (1+ |\xi|^2) + \dfrac{1}{2}\log (1 + |\eta|^2),
\end{align*}
we estimate
\begin{align}
    \int_{\R^d}WK * W \, \mathrm{d}\xi &= -\int_{\R^d}\int_{\R^d}W(\xi)W(\eta)\log |\xi-\eta| \, \mathrm{d}\eta \, \mathrm{d}\xi\notag\\
    & \ge  - \int_{\R^d}\int_{\R^d}W(\xi)W(\eta)\log (1 + |\xi|^2) \, \mathrm{d}\eta \, \mathrm{d}\xi\notag\\
    &\ge  - CM \int_{\R^d}|\xi|^2 W  \, \mathrm{d}\xi.\notag
\end{align}
Thus, in view of the second moment bound \eqref{eq2:estimate for W}, it follows that, for every $s\geq s_0$,
\begin{align}\label{eq2:Lyapunov for W}
\int_{\R^d}WK*W \, \mathrm{d}\xi \ge C
\end{align}
with a number $C\in \R$ depending only on the mass $M$, the dimension $d\geq 2$, and the initial moment \eqref{initial_moment}.

Finally, combining inequalities \eqref{eq1: Lyapunov for W} and \eqref{eq2:Lyapunov for W}, we conclude that the Lyapunov functional $\mathcal{F}$ is bounded from below for all $s\geq s_0$.
\end{proof}

We apply the LaSalle principle as stated in Remark~\ref{LaSalle} to
the dynamical system $S(\tau)$ induced by  the rescaled equation \eqref{eq:FP}.
For every $M > 0$ and some $q \in (1,\infty)$, set 
\begin{align}\label{XM:metric-space}
    \mathcal{X}_M := \{f \in L^1_2(\mathbb R^d)\cap L^q(\mathbb R^d): f \ge 0,\quad \int_{\R^d} f(\xi)\, \mathrm{d}\xi = M\},
\end{align}
where
\begin{align}\label{eq:L12-space}
L^1_2(\mathbb R^d)
:= \{
f\in L^1(\mathbb R^d):
\int_{\mathbb R^d}(1+|\xi|^2)|f(\xi)|\, \mathrm{d}\xi<\infty\}.
\end{align}
We endow $\mathcal{X}_M$ with the metric
\begin{align}\label{eq:Y-metric}
d_{\mathcal{X}_M}(f,g)
:= \|f-g\|_{1} + \|f-g\|_q + \int_{\R^d}|\xi|^2|f(\xi)-g(\xi)|\, \mathrm{d}\xi .
\end{align}
Then $\mathcal X_M$ is a complete metric space.

The rescaled equation \eqref{eq:FP} generates a continuous dynamical system
$S(\tau):\mathcal X_M\to\mathcal X_M$ for each $\tau \ge 0$. Let $W$ be the rescaled solution
defined by \eqref{eq:def-W}. For fixed $s_0\in\mathbb R$, we have
\begin{align*}
S(\tau)W(s_0)=W(s_0+\tau)
\qquad\text{for every }\tau\ge 0 .
\end{align*}
Hence the positive trajectory of $W(s_0)$ is
\begin{align*}
    \bigcup_{\tau\ge0}\{S(\tau)W(s_0)\}
    =
    \{W(s):s\ge s_0\}.
\end{align*}
Moreover, the trajectory is relatively compact in $\mathcal X_M$ by the compactness estimates and the convergence of the second moment established in Corollary~\ref{prop:compact for W}.

The Lyapunov functional $\mathcal F$ given by Lemma~\ref{lem:Lyapunov for W}
is continuous with respect to the $\mathcal X_M$-topology and also is nonincreasing along trajectories. Therefore, the LaSalle principle stated in Remark~\ref{LaSalle} applies to the dynamical system $S(\tau)$ on $\mathcal{X}_M$.

\begin{lemma}\label{lemma:omega-energy-space:new}
Let $W$ be defined by \eqref{eq:def-W},
where $u = u(x,t)$ is a solution to problem \eqref{eq:main}--\eqref{ini:L1Linf} with the logarithmic kernel \eqref{K=logx} and with
a nonnegative initial condition satisfying
$
\int_{\mathbb{R}^d} |x|^2 u_0(x) \, \mathrm{d}x < \infty.
$
Then $\omega(W)\neq \emptyset$ and  
$$\omega(W)\subset \mathcal{X}_M \cap \left(\bigcap_{1 \le p < \infty} L^p(\mathbb{R}^d)\right).$$
\end{lemma}

\begin{proof}
Since the trajectory is relatively compact in $\mathcal X_M$, the $\omega$-limit set $\omega(W)$ is nonempty. Let $W_0 \in \omega (W)$, then there exist a sequence
\(\{s_n\}\) such that
\begin{equation}\label{eq:Wn-L1-convergence}
W(\cdot,s_n)\longrightarrow W_0
\qquad\text{in }\mathcal{X}_M.
\end{equation}
After passing to a further subsequence, we may also assume that
\[
W(\xi,s_n)\longrightarrow W_0(\xi)
\qquad\text{for a.e. }\xi\in\mathbb R^d.
\]
By applying the Fatou lemma
and the $L^p$-bounds in \eqref{eq1:estimate for W}, it is routine to show that $W_0 \in \bigcap_{1 \le p < \infty} L^p(\mathbb{R}^d)$.
\end{proof}

The following lemma establishes the first step toward applying the LaSalle principle recalled in Remark~\ref{LaSalle}.

\begin{lemma}\label{lemma:stationaly for W}
Let \(W\) be defined by \eqref{eq:def-W}, where \(u=u(x,t)\) is a
solution of \eqref{eq:main}--\eqref{ini:L1Linf} with the logarithmic
kernel \eqref{K=logx}. Assume that
\(
\int_{\mathbb R^d}|x|^2u_0(x)\, \mathrm{d}x<\infty.
\)
Then every $W_0\in\omega(W)$ is a stationary solution of \eqref{eq:FP} in the weak sense.
In particular, $\mathcal D[W_0]=0$.
\end{lemma}

\begin{proof}
Let $W_0\in\omega(W)$. By the LaSalle principle applied in the complete
metric space $\mathcal X_M$ to the relatively compact positive trajectory of
$W$, we have
\begin{align}\label{eq:F-constant-along-W0}
    \mathcal F(S(\tau)W_0)=\mathcal F(W_0)
    \qquad\text{for every } \tau\ge0,
\end{align}
where the functional $\mathcal{F}$ is established in Lemma~\ref{lem:Lyapunov for W}. Lemma~\ref{lemma:omega-energy-space:new} gives
$W_0\in\mathcal X_M\cap\bigcap_{1\le p<\infty}L^p(\mathbb R^d)$ and in particular, $W_0\ge0$ in $\R^d$ and $\int_{\mathbb R^d}W_0\,\mathrm{d}\xi=M>0$.
Hence the parabolic regularity theorem and the strong maximum principle for
\eqref{eq:FP} imply that $S(\tau)W_0$ is regular and positive in $\mathbb R^d$ for every $\tau>0$. 
Let $0<a<b$. Since $S(\tau)W_0$ is sufficiently regular for $\tau>0$,
the Lyapunov identity in Lemma~\ref{lem:Lyapunov for W} is valid on $[a,b]$.
Using \eqref{eq:F-constant-along-W0}, we obtain
\begin{align*}
    0
    =
    \mathcal F(S(b)W_0)-\mathcal F(S(a)W_0)
    =
    -\int_a^b \mathcal D[S(\tau)W_0]\,\mathrm{d}\tau,
\end{align*}
which implies that
\begin{align}\label{eq:D-zero-ae}
    \mathcal D[S(\tau)W_0]=0
    \qquad\text{for a.e. } \tau>0 .
\end{align}
Thus the definition of $\mathcal{D}[S(\tau)W_0]$ in Lemma~\ref{lem:Lyapunov for W}
and the positivity of $S(\tau)W_0$ allow us to obtain the zero-flux identity
\begin{align}\label{eq:zero-flux-W0-shifted}
    \nabla S(\tau)W_0
    +
    \frac{\xi}{2}S(\tau)W_0
    +
    S(\tau)W_0\bigl(\nabla K*S(\tau)W_0\bigr)
    =
    0
\end{align}
in the sense of distributions. Consequently $S(\tau)W_0$ is a weak stationary
solution of \eqref{eq:FP}.

We now show that $W_0$ itself is stationary solution of \eqref{eq:FP}. Fix $t>0$ and choose a sequence
$\tau_n\downarrow0$ such that \eqref{eq:zero-flux-W0-shifted} holds with
$\tau=\tau_n$ and $\tau_n<t$. Since $S(\tau_n)W_0$ is stationary, we have
\begin{align*}
    S(t-\tau_n)S(\tau_n)W_0
    =
    S(\tau_n)W_0 .
\end{align*}
The semigroup property of the dynamical system gives
\begin{align*}
    S(t)W_0
    =
    S(\tau_n)W_0\qquad \text{for every } t > 0.
\end{align*}
Letting $n\to\infty$ and using the continuity of the map $\tau\mapsto S(\tau)W_0$ in $\mathcal X_M$, we obtain
\begin{align*}
    S(t)W_0=W_0.
\end{align*}
Since $t>0$ is arbitrary, we note that $W_0$ is a stationary solution of \eqref{eq:FP}.

Finally, since $W_0=S(t)W_0$ for every $t>0$, the parabolic regularity of
$S(t)W_0$ implies that $W_0$ is sufficiently regular. Therefore the dissipation
$\mathcal D[W_0]$ is well-defined, and from \eqref{eq:D-zero-ae} we also obtain
\begin{align*}
    \mathcal D[W_0]=0 .
\end{align*}
This completes the proof.
\end{proof}



\subsection{Uniqueness of the self-similar profile}
It remains to prove the uniqueness of stationary solutions to the rescaled
equation \eqref{eq:FP} in the class obtained above. More precisely, we show
that \eqref{eq:FP} admits at most one nonnegative weak stationary solution in
\begin{align}\label{domain:XM}
\mathcal X_M
\cap
\left(\bigcap_{1\le p<\infty}L^p(\mathbb R^d)\right).
\end{align}
Owing to Lemma~\ref{lemma:omega-energy-space:new} and Lemma~\ref{lemma:stationaly for W}, every element of $\omega(W)$ belongs to the class
\eqref{domain:XM} and is a stationary weak solution of \eqref{eq:FP}. Hence,
once the uniqueness is proved, all omega-limit points of $W$ coincide with the
same stationary profile.




To proceed, we establish the following inequality for the logarithmic interaction energy. The sign of the quadratic form below plays a crucial role in the uniqueness argument.

\begin{lemma}\label{lem:fourier inequality}
Let $h \in L^1(\mathbb{R}^d) \cap L^p(\mathbb{R}^d)$ for some $p \in (1,\infty)$. Assume that
\[
\int_{\mathbb{R}^d} h(x) \, \mathrm{d}x = 0
\quad
\text{and}
\quad
\int_{\mathbb{R}^d} \log(1+|x|^2) |h(x)| \, \mathrm{d}x < \infty.
\]
Then
\begin{equation}\label{log:inequality}
-\iint_{\mathbb{R}^d \times \mathbb{R}^d} \log|x-y| \, h(x)h(y) \, \mathrm{d}y \, \mathrm{d}x \ge 0.
\end{equation}
\end{lemma}

\begin{proof}
First assume that $h \in \mathcal{S}(\mathbb{R}^d)$ (the Schwartz class) with $\int_{\mathbb{R}^d} h(x) \, \mathrm{d}x = 0$. We adopt the Fourier transform convention
\[
\widehat{f}(\xi) = \int_{\mathbb{R}^d} f(x) e^{-2\pi i x \cdot \xi} \, \mathrm{d}x.
\]
In the sense of tempered distributions, the Fourier transform of the logarithmic kernel satisfies
\[
\widehat{-\log|x|} = A_d \, \mathrm{FP}(|\xi|^{-d}) + B_d \delta_0, \qquad A_d > 0.
\]
Since
\[
\widehat{h}(0) = \int_{\mathbb{R}^d} h(x) \, \mathrm{d}x = 0,
\]
the Dirac mass does not contribute. Moreover, because $\widehat{h}(\xi) = O(|\xi|)$ as $\xi \to 0$, the finite-part distribution reduces to an ordinary integral when evaluated on $|\widehat{h}|^2$. Hence,
\begin{equation}\label{eq:Fourier-log-energy}
-\iint_{\mathbb{R}^d \times \mathbb{R}^d} \log|x-y| \, h(x)h(y) \, \mathrm{d}y \, \mathrm{d}x = A_d \int_{\mathbb{R}^d} \frac{|\widehat{h}(\xi)|^2}{|\xi|^d} \, \mathrm{d}\xi \ge 0.
\end{equation}

We now consider a general function $h$ satisfying the assumptions of the lemma. Choose a sequence $h_n \in \mathcal{S}(\mathbb{R}^d)$ such that
\[
\int_{\mathbb{R}^d} h_n(x) \, \mathrm{d}x = 0
\]
and
\begin{equation}\label{eq:approx-log-h}
\|h_n - h\|_1 + \|h_n - h\|_p + \int_{\mathbb{R}^d} \log(1+|x|^2) |h_n(x) - h(x)| \, \mathrm{d}x \to 0 \qquad \text{as } n \to \infty.
\end{equation}
It remains to pass to the limit in the logarithmic interaction energy.
Set
\[
k_0(z) = |\log|z|| \, \mathbf{1}_{\{|z|<1\}}, \qquad k_\infty(z) = \log|z| \, \mathbf{1}_{\{|z|\ge1\}}.
\]
For the singular part, we have $k_0 \in L^r(\mathbb{R}^d)$ for every $1 \le r < \infty$. Hence, the  Young and the H\"older inequalities imply that
\[
\iint_{|x-y|<1} |\log|x-y|| \, |h_n(x) - h(x)| \, |h_n(y)| \, \mathrm{d}y \, \mathrm{d}x \to 0 \qquad \text{as } n \to \infty.
\]
Indeed, if $1 < p \le 2$, one takes
\[
r = \frac{p}{2(p-1)},
\]
whereas if $p > 2$, one takes $r = p'$ along with the $L^1$-convergence in \eqref{eq:approx-log-h}.

For the far-field part, we use the bound
\[
0 \le \log|x-y| \le \frac{1}{2}\log(1+|x|^2) + \frac{1}{2}\log(1+|y|^2) \qquad \text{for } |x-y| \ge 1.
\]
Therefore, relation \eqref{eq:approx-log-h} implies that
\[
\iint_{|x-y|\ge1} \log|x-y| \, |h_n(x) - h(x)| \, |h_n(y)| \, \mathrm{d}y \, \mathrm{d}x \to 0 \qquad \text{as } n \to \infty.
\]
The same estimates hold with $h_n(y) - h(y)$ in place of $h_n(x) - h(x)$. Consequently,
\[
\iint_{\mathbb{R}^d \times \mathbb{R}^d} \log|x-y| \, h_n(x)h_n(y) \, \mathrm{d}y \, \mathrm{d}x \to \iint_{\mathbb{R}^d \times \mathbb{R}^d} \log|x-y| \, h(x)h(y) \, \mathrm{d}y \, \mathrm{d}x.
\]
Applying inequality \eqref{eq:Fourier-log-energy} to $h_n$ and passing to the limit yields inequality \eqref{log:inequality}.
\end{proof}

We are now ready to prove the uniqueness of stationary solutions to equation \eqref{eq:FP}. Here, the integrability
assumption in \eqref{domain:XM} allows us to choose an exponent larger than
$d$, which will be used to obtain higher regularity of stationary solutions.

\begin{theorem}[Uniqueness of the self-similar profile]
\label{thm:uniqueness-stationary-profile}
Let \(M>0\). Suppose that
\begin{equation}\label{U1U2}
U_1,U_2
\in
\mathcal{X}_M\cap
\bigcap_{1\leq p<\infty}L^p(\mathbb R^d)
\end{equation}
are weak solutions of
\begin{equation}\label{eq:stationary-zero-flux}
\Delta U
+\frac12\nabla\cdot(\xi U)
+\nabla\cdot\bigl(U(\nabla K*U)\bigr)
=0
\qquad\text{in }\mathbb R^d,
\end{equation}
where
\(
K(x)=-\log|x|.
\)
Then
\(
U_1=U_2
\)
a.e. in $\mathbb R^d$.
In particular, equation \eqref{eq:stationary-zero-flux} admits at
most one solution in the class \eqref{U1U2}.
\end{theorem}

\begin{proof}
Let \(U\) be a weak solution of \eqref{eq:stationary-zero-flux}
belonging to the class \eqref{U1U2}. We first show that \(U\) is
continuous.
Combining the estimates from Lemma \ref{kernel-estimate} with the H\"older inequality, 
 for every finite \(p\), we obtain
\(
U(\nabla K*U)\in L^p(\mathbb R^d)
\)
and likewise
\(
\xi U\in L^p_{\mathrm{loc}}(\mathbb R^d).
\)
Equation \eqref{eq:stationary-zero-flux} can therefore be written as
\[
\Delta U
=
-\nabla\cdot
\left(
\frac{\xi}{2}U+U(\nabla K*U)
\right),
\]
where the vector field on the right-hand side belongs to
\(L^p_{\mathrm{loc}}(\mathbb R^d)\). Standard local elliptic
regularity yields
\(
U\in W^{1,p}_{\mathrm{loc}}(\mathbb R^d)
\)
for every $p<\infty$.
Choosing  \(p>d\) and applying the Morrey embedding, we conclude that $U$ is H\"older continuous.
 In particular, \(U\) is continuous.

We shall also use the fact that \(K*U\) is continuous. Indeed, the
singularity of \(\log|x|\) is locally integrable, while the finite
second moment implies
\[
\int_{\mathbb R^d}
\log(1+|x|^2)U(x)\,\mathrm{d}x<\infty.
\]
Splitting the convolution into a neighborhood of the singularity
and its complement gives
\(
K*U\in C(\mathbb R^d).
\)

A stationary solution \(U\) defines the constant trajectory
\(W(\cdot,s)=U\) of equation \eqref{eq:FP}. Since \(U\in \mathcal{X}_M\) and
belongs to all finite \(L^p\)-spaces, the  identity
\eqref{energy-dissipation-rescal} applies and gives
\begin{equation}\label{eq:D-U-zero}
\mathcal D[U]=0.
\end{equation}

Set
\[
\Omega:=\{\xi\in\mathbb R^d:U(\xi)>0\}.
\]
Since \(U\) is continuous, \(\Omega\) is open. It is nonempty because
\(
\int_{\mathbb R^d}U(\xi)\,\mathrm{d}\xi=M>0.
\)
Let \(\Omega_0\) be a connected component of \(\Omega\).
By relation \eqref{eq:D-U-zero}, it follows that 
\[
\nabla
\left(
\log U+\frac{|\xi|^2}{4}+K*U
\right)
=0
\qquad\text{a.e. in }\Omega_0.
\]
Therefore, there exists \(c_0\in\mathbb R\) such that
\begin{equation}\label{eq:positivity for U}
\log U(\xi)
+\frac{|\xi|^2}{4}
+(K*U)(\xi)
=
c_0
\qquad\text{for }\xi\in\Omega_0.
\end{equation}
The identity holds pointwise because the terms involved are
continuous on \(\Omega_0\).

We claim that \(\Omega_0=\mathbb R^d\). Suppose otherwise. Then there
exists
\(
\xi_0\in\partial\Omega_0.
\)
Choose \(\xi_n\in\Omega_0\) such that \(\xi_n\to\xi_0\). Since \(U\)
is continuous and \(\xi_0\notin\Omega_0\),
\(
U(\xi_n)\longrightarrow U(\xi_0)=0,
\)
and hence
\(
\log U(\xi_n)\longrightarrow-\infty.
\)
On the other hand, relation \eqref{eq:positivity for U} gives
\[
\log U(\xi_n)
=
c_0-\frac{|\xi_n|^2}{4}-(K*U)(\xi_n).
\]
Since \(K*U\) is continuous, the right-hand side has a finite limit
as \(n\to\infty\), which is a contradiction. Thus
\(
\Omega_0=\mathbb R^d;
\)
 consequently,
\(
U(\xi)>0
\)
for every $\xi\in\R^d$,
and there exists \(c\in\R\) such that
\begin{equation}\label{eq:EL-stationary}
\log U(\xi)
+\frac{|\xi|^2}{4}
+(K*U)(\xi)
=
c
\qquad\text{for every }\xi\in\mathbb R^d.
\end{equation}

Let \(U_1\) and \(U_2\) satisfy the assumptions of the theorem.
We have already proved that
 there exist constants \(c_1,c_2\in\mathbb R\) such that
\[
\log U_i
+\frac{|\xi|^2}{4}
+K*U_i
=
c_i,
\qquad i=1,2.
\]
Since \(U_1\) and \(U_2\) have the same mass, the function
\[
h:=U_1-U_2
\quad \text{satisfies}\quad 
\int_{\mathbb R^d}h(\xi)\,\mathrm{d}\xi=0.
\]
Subtracting the two identities for $U_i$ gives
\begin{equation}\label{eq:difference-EL}
\log U_1-\log U_2+K*h=c_1-c_2.
\end{equation}
Multiplying this equation by \(h\), integrating over $\R^d$, using  zero mass of $h$, and  
 \(K(x)=-\log|x|\),
we
obtain
\begin{equation}\label{eq:key-identity-uniqueness-log}
\int_{\mathbb R^d}
(U_1-U_2)(\log U_1-\log U_2)\,\mathrm{d}\xi
=
\iint_{\mathbb R^d\times\mathbb R^d}
\log|\xi-\eta|\,h(\xi)h(\eta)\,\mathrm{d}\eta\,\mathrm{d}\xi.
\end{equation}
The function $h$ satisfies 
 the assumptions of Lemma~\ref{lem:fourier inequality}, thus the right-hand side of equation \eqref{eq:key-identity-uniqueness-log}
 is non-positive.
On the other hand, since \(r\mapsto\log r\) is strictly increasing
on \((0,\infty)\),
 the left-hand side of equation
\eqref{eq:key-identity-uniqueness-log} is nonnegative.
This is possible only if 
\[
\int_{\mathbb R^d}
(U_1-U_2)(\log U_1-\log U_2)\,\mathrm{d}\xi=0.
\]
Strict monotonicity of the logarithm implies
\(
U_1(\xi)=U_2(\xi)
\)
a.e.~on $\R^d$.
Since both functions are continuous, they agree everywhere in
\(\R^d\). 
\end{proof}


\subsection{Self-similar large time behavior}

Using the tools introduced in the previous subsection, we are now in a position to obtain the large-time self-similar behavior of solutions to problem \eqref{eq:main}--\eqref{ini:L1Linf} with the critical interaction kernel.

\begin{proof}[Proof of Theorem~\ref{main thm:self-similar}]
We begin the proof by applying the LaSalle invariance principle recalled in Remark~\ref{LaSalle} to the dynamical system generated by equation \eqref{eq:FP} in the metric space $\mathcal{X}_M$ defined in \eqref{XM:metric-space}. As a consequence of Lemma~\ref{lemma:stationaly for W} and Theorem~\ref{thm:uniqueness-stationary-profile}, combined with the LaSalle principle, we obtain that the solution $W(s)$ with mass $M > 0$ to equation \eqref{eq:FP} converges in $L^1(\mathbb{R}^d)$ toward the unique stationary solution $\Phi_M$ with mass $M$ belonging to class \eqref{domain:XM}; namely,
\begin{equation}
    \|W(s)-\Phi_M\|_1 \to 0 \qquad \text{as } s \to \infty.
\end{equation}
Combining this fact with the uniform bound for $L^r$-norms in \eqref{eq1:estimate for W} and the H\"older inequality, we obtain
\begin{equation}\label{convergence:W:Lp}
    \|W(s)-\Phi_M\|_p \to 0 \qquad \text{as } s \to \infty
\end{equation}
for each $p \in [1,\infty)$. Using the uniform gradient estimate in \eqref{eq1:estimate for W} for $p > d$ and the standard 
estimate of the 
$L^\infty$-norm by the Sobolev norm, we immediately obtain the convergence \eqref{convergence:W:Lp} for $p = \infty$ as well.

We now translate this convergence result to solutions to equation \eqref{eq:main} using the scaling relations \eqref{eq:def-W} and \eqref{eq:inverse-W}.
Define
\[
    U_M(x,t)
    :=
    t^{-\frac d2}\Phi_M\left(\frac{x}{\sqrt t}\right).
\]
Then, by the change of variables \(x=\sqrt t\,\xi\) and the convergence \eqref{convergence:W:Lp}, we obtain
\begin{align}\label{proof1:main thm self similar}
    t^{\frac d2(1-\frac1p)}
    \|u(t)-U_M(t)\|_{p}
    =
    \left\|
    W(\cdot,\log t)-\Phi_M
    \right\|_{p}
    \to0\quad\text{as } t \to \infty
\end{align}
for every \(p\in[1,\infty]\).
\end{proof}


\section*{Declarations}

\textbf{Data Availability.}
Data sharing is not applicable to this article since no datasets were generated or analyzed in this study.

\textbf{Conflict of interest.}
The authors declare that there are no conflicts of interest.


\bibliographystyle{siam}
\bibliography{biblio}


\end{document}